\documentclass[pdflatex,sn-basic,Numbered]{sn-jnl}

\usepackage{graphicx}
\usepackage{xspace}
\usepackage{bm}
\usepackage{amsmath, amsthm, amssymb}
\usepackage{enumitem}
\usepackage{booktabs}
\usepackage{array}
\usepackage{multirow}
\usepackage{caption}
\usepackage{subcaption}
\usepackage{xcolor}
\usepackage{manyfoot}

\newtheorem{theorem}{Theorem}
\newtheorem{lemma}[theorem]{Lemma}

\newtheorem{proposition}[theorem]{Proposition}

\theoremstyle{definition}

\newtheorem{remark}[theorem]{Remark}

\title[Graph-Based ECG Synthesis]{Graph-Based ECG Synthesis with Activation-Consistency Certification and Diagnostics-Aware Morphology Curation}

\author*[1]{\fnm{Yongjun} \sur{Yoo}}\email{yooyj228@snu.ac.kr}
\affil*[1]{\orgdiv{Department of Mathematical Science}, \orgname{Seoul National University}, \orgaddress{\city{Seoul}, \country{South Korea}}}

\abstract{Synthetic electrocardiogram (ECG) generation can support algorithm development and robustness evaluation, but simulated signals must preserve interpretable activation, recovery, and morphology properties. We present a graph-based ECG synthesis framework that combines activation-consistency certification with diagnostics-aware morphology curation. A unified heart graph supports an eikonal-template backend (ET) and a pseudo-diffusion reaction--eikonal backend (RE). We formulate graph Eikonal activation as a Bellman fixed-point problem and use the Bellman residual as a computable certificate for activation-time consistency. Each simulated ECG is evaluated by a two-stage diagnostics pipeline that separates metric computation from experiment-specific acceptance policies. On the cardiac graph, RE-derived activation times showed near-millisecond agreement with the Eikonal backbone and achieved $R^2=0.99876$ after causal predecessor filtering. Recovery experiments showed that endo-epicardial APD gradients determined the main T-wave morphology window, whereas the diffusion strength~$\kappa$ provided secondary repolarization smoothing. In final balanced multi-lead curation, RE accepted 658/2000 samples versus 578/2000 for ET and increased per-model morphology coverage from 0.09248 to 0.09888. The framework provides a conservative basis for controllable and curated synthetic ECG generation.}


\begin{document}

\maketitle

\section{Introduction}
\label{sec:intro}

Synthetic electrocardiogram (ECG) generation is increasingly useful for biomedical signal-processing research. Curated synthetic ECGs can support algorithm development, robustness evaluation, rare-morphology evaluation, and data augmentation when real recordings are limited by privacy constraints, annotation cost, class imbalance, or incomplete coverage of clinically relevant waveform patterns~\cite{Kligfield2007ECGStandardization, Rautaharju2009STTQT}. However, synthetic ECGs are useful only when their morphology remains interpretable. A generated trace that is numerically stable but lacks plausible activation order, P-QRS-T structure, repolarization behavior, or lead-wise consistency can mislead downstream evaluation rather than improve it.

Existing ECG synthesis approaches face a trade-off between realism, controllability, and quality control. Recent reviews place current methods into mathematical, computer-vision-inherited, and deep generative families, while also noting the lack of standardized quality-assessment criteria across synthetic ECG studies~\cite{Zanchi2025ScopingReview}. Data-driven generative models, including GAN-based, conditional, and diffusion-based ECG generators, can learn complex waveform distributions and support augmentation or semantically guided synthesis~\cite{Zhang2021GAN12Lead, Xia2023ConditionalECG, Lin2025TransDiffECG}, but their latent factors are often difficult to connect to activation and recovery mechanisms, and additional screening is still required to ensure that generated signals satisfy physiological and diagnostic constraints. In contrast, computational electrophysiology models provide interpretable parameters and mechanistic structure, but large-scale sampling can produce many numerically valid yet morphologically implausible ECGs. Fast eikonal and reaction--eikonal methods are attractive for scalable simulation because they separate activation timing from local waveform generation~\cite{Neic2017ReactionEikonal, vanDam2009ActivationRecovery}, but synthetic ECG generation still requires explicit checks that activation fields are consistent and that accepted signals meet morphology criteria.

This study addresses these issues with a graph-based ECG synthesis framework designed around two quality-control layers. We choose a graph formulation to keep the propagation workflow simpler, faster, and easier to reproduce than a full mesh-based partial-differential-equation simulation, while retaining enough anatomical structure for controlled ECG synthesis experiments. The first layer concerns activation consistency. We formulate graph eikonal activation as a Bellman fixed-point problem and use the associated residual as a computable certificate for activation-time fields. This certificate is not intended to replace full reaction--diffusion validation; rather, it provides a practical way to test whether an extracted or surrogate activation field remains consistent with the graph-based activation backbone used by the ECG generator.

The second layer concerns ECG signal morphology. We build a unified heart graph and shared 12-lead forward pipeline that support two final generators: an eikonal-template backend~(\textrm{ET}) and a pseudo-diffusion reaction--eikonal backend~(\textrm{RE}). The heart graph is intentionally simplified so that the full synthesis pipeline remains transparent and reproducible rather than tied to a heavy subject-specific preprocessing stack. The parameterization is organized into activation and recovery controls so that activation-atlas, recovery-atlas, and sensitivity experiments can test whether interpretable knob-response behavior is preserved. Each generated ECG is then evaluated by a diagnostics-aware curation pipeline that separates metric computation from experiment-specific acceptance policies. This separation allows broad throughput screening, recovery-mechanism analysis, sensitivity testing, and final balanced multi-lead morphology curation to use the same underlying diagnostics without treating all acceptance rates as equivalent.

The present work focuses on the forward synthesis direction: activation and recovery fields are specified on the cardiac graph and projected through a fixed heart--torso--electrode configuration to generate ECG leads. A complementary inverse problem--estimating latent cardiac activation or recovery structure from observed ECGs--is therefore left as future work. The graph formulation, Bellman consistency certificate, and diagnostics-aware acceptance rules developed here could provide useful constraints or initialization signals for such inverse modeling, but inverse estimation is not evaluated in this study.

\begin{figure}[t]
  \centering
  \begin{subfigure}[t]{0.31\linewidth}
    \centering
    \includegraphics[width=0.95\linewidth]{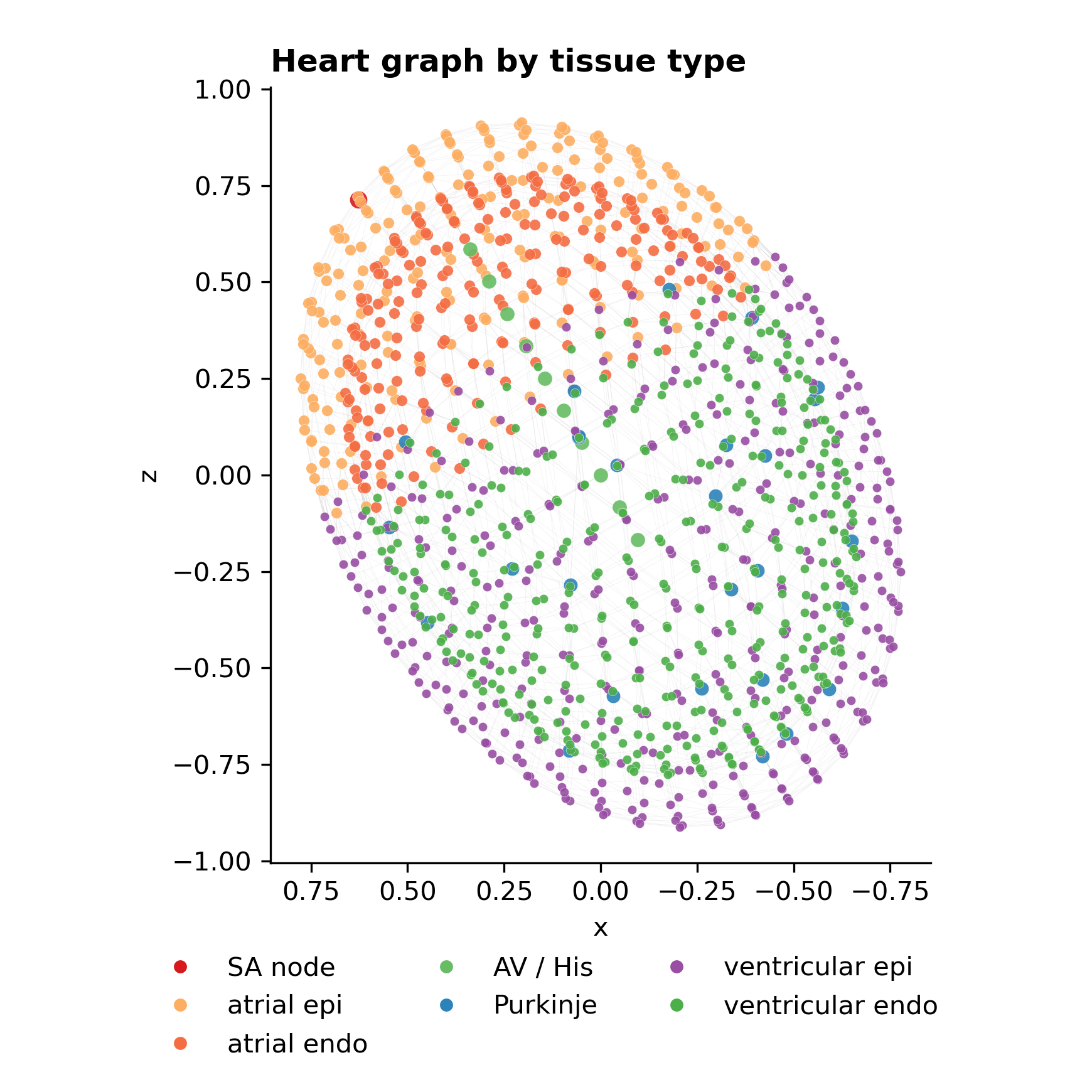}
    \caption{Tissue-colored cardiac graph.}
    \label{fig:heart_graph_tissue_context}
  \end{subfigure}
  \hfill
  \begin{subfigure}[t]{0.31\linewidth}
    \centering
    \includegraphics[width=0.95\linewidth]{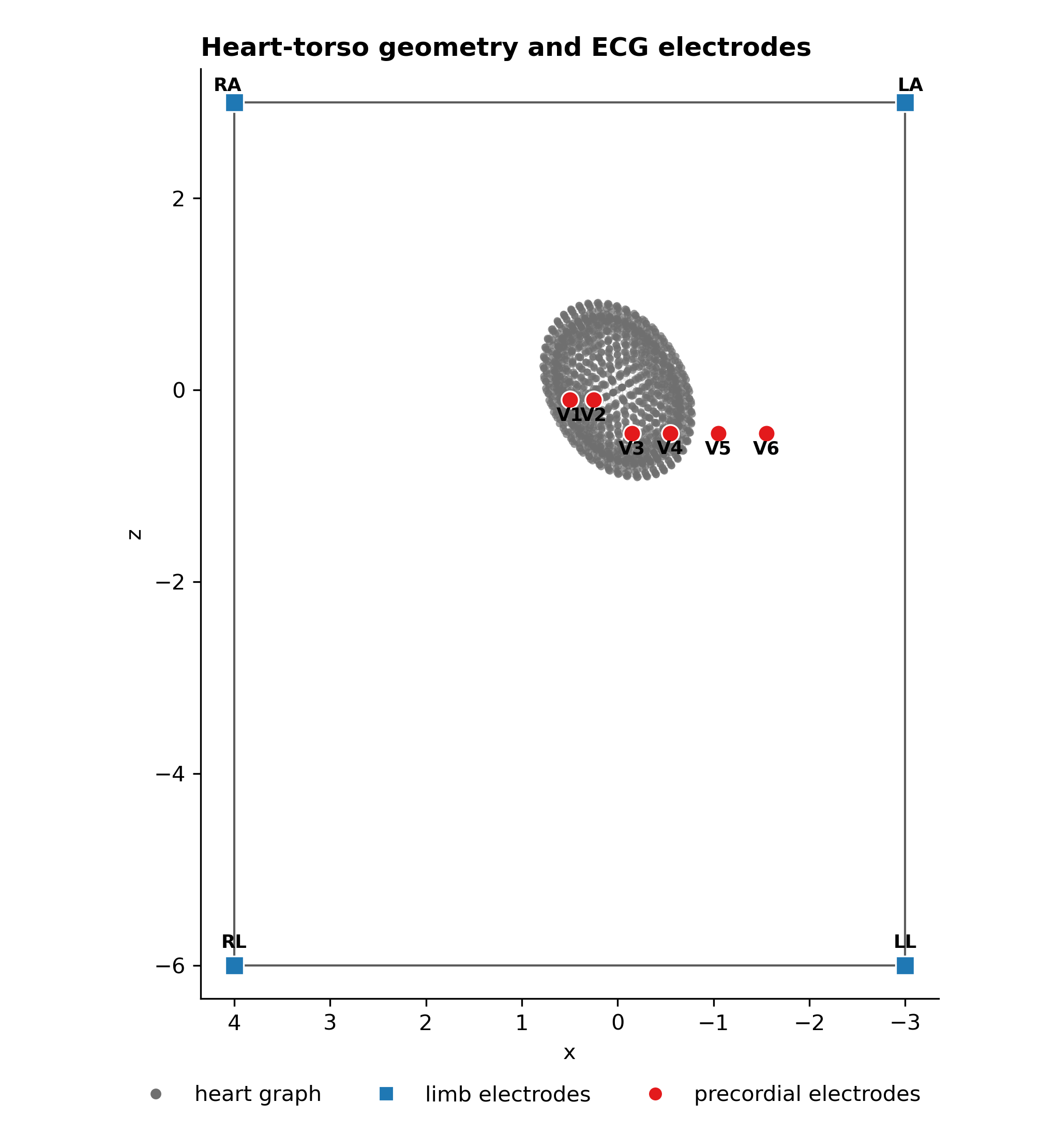}
    \caption{Heart--torso graph with ECG electrodes.}
    \label{fig:heart_torso_electrodes_context}
  \end{subfigure}
  \hfill
  \begin{subfigure}[t]{0.31\linewidth}
    \centering
    \includegraphics[width=0.95\linewidth]{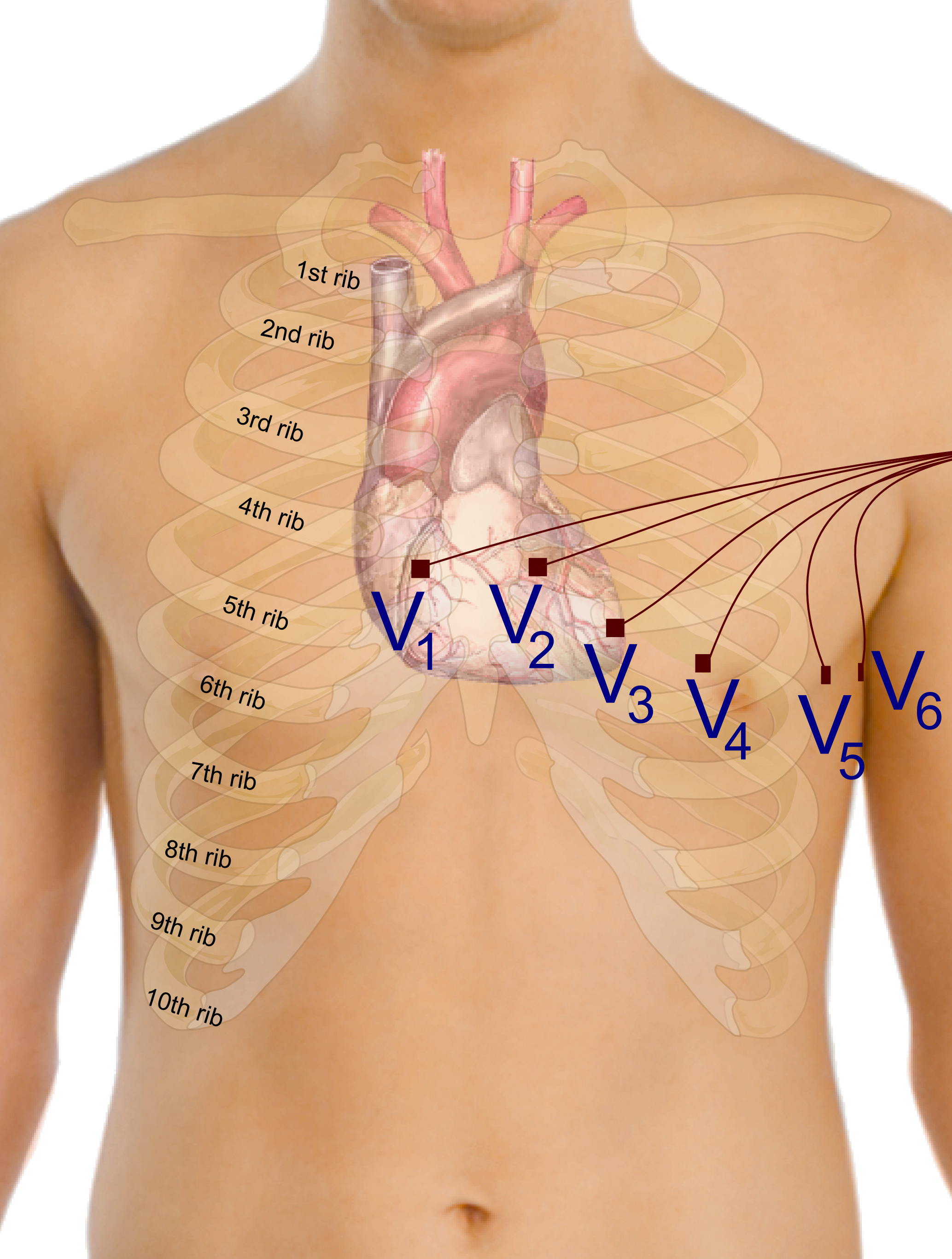}
    \caption{Precordial lead anatomy.}
    \label{fig:precordial_leads_context}
  \end{subfigure}
  \caption{Graph-based geometry and anatomical lead-placement context for the forward ECG synthesis pipeline. Panels (a) and (b) show the simplified experimental geometry; panel (c) provides anatomical context for the precordial lead layout. Panel (c) image source: Mikael Haggstrom, public-domain CC0 illustration from Wikimedia Commons.}
  \label{fig:geometry_lead_context}
\end{figure}

The main contributions of this study are:
\begin{enumerate}[leftmargin=*]
  \item We formulate graph eikonal activation as a Bellman fixed-point problem and derive a computable residual-based certificate for activation-time consistency.
  \item We develop a unified graph-based ECG synthesis pipeline with \textrm{ET} and \textrm{RE} backends, separated activation and recovery controls, and a shared 12-lead ECG forward model.
  \item We introduce a diagnostics-aware curation pipeline that separates metric computation from experiment-specific acceptance policies for throughput screening, recovery analysis, sensitivity testing, and final morphology curation.
  \item We validate the framework on cardiac ECG synthesis experiments, showing preserved activation controllability and a modest \textrm{RE} advantage in final morphology curation, including higher accepted yield (658/2000 vs. 578/2000) and feature-space coverage (0.09888 vs. 0.09248). As an application-oriented evaluation, we also test whether biophysically motivated synthetic ECGs provide a stabilizing initialization signal for low-label P/QRS/T delineation on PTB-XL+, showing that parameter-informed pretraining can reduce model-collapse risk and variance under extreme label scarcity.
\end{enumerate}

% ======================================================================
\section{Methods}
\label{sec:methods}
% ======================================================================

% ----------------------------------------------------------------------
\subsection{Unified heart graph and forward pipeline}
\label{sec:heart_graph}
% ----------------------------------------------------------------------

All simulations share a common computational substrate: a unified heart graph $G=(V,E)$ with dual-layer atrial and ventricular structures, a specialized cardiac conduction system, and atrioventricular insulation through a fibrous annulus~\cite{Padala2021CardiacConductionSystem, Dobrzynski2013CardiacConductionSystem, Saremi2017FibrousSkeleton}. Each node represents a spatial point in the cardiac anatomy, and each weighted edge represents a propagation pathway with tissue-specific
conductivity. The graph includes node sets for sinoatrial, atrial (left/right, endo/epi), atrioventricular, His-bundle, Purkinje (left/right), and ventricular (endo/epi, LV/RV) tissue types.

Both solver backends described below operate on the same graph and use the same parameters for geometry, tissue labels, source placement, and measurement configuration. The torso geometry, heart placement, rotation convention, and 12-lead definition are fixed across all experiments.

After the transmembrane potential field~$V_m$ is generated by either backend, the same graph-based forward pipeline is used to produce the 12-lead ECG. First, the cardiac extracellular potential~$\phi_e$ is recovered from~$V_m$ through a quasi-static relation on the heart graph:
\begin{equation}\label{eq:phi_e}
  L_{\mathrm{cond}}\,\phi_e
  \;=\;
  -\,L_{\mathrm{prop}}\,V_m,
\end{equation}
where $L_{\mathrm{prop}}$ and $L_{\mathrm{cond}}$ are weighted graph Laplacian operators derived from the shared conductivity structure~\cite{Neic2017ReactionEikonal, Rudy2015ForwardProblem}. The graph Laplacian represents conductivity-weighted potential differences across neighboring nodes; here it converts spatial variation in~$V_m$ into source terms and distributes the resulting extracellular potential over the cardiac graph.

Second, let $\phi_K$ denote the restriction of~$\phi_e$ to the cardiac boundary nodes coupled to the torso graph. If $L_{UU}$ and $L_{UK}$ are the torso-graph Laplacian blocks corresponding respectively to unknown torso nodes~$U$ and their coupling to the cardiac boundary nodes~$K$, the torso potential~$u_T$ is obtained by solving
\begin{equation}\label{eq:torso_forward}
  L_{UU}\,u_T
  \;=\;
  -\,L_{UK}\,\phi_K .
\end{equation}
Electrode potentials sampled from~$u_T$ are then projected onto the standard 12-lead configuration~\cite{Kligfield2007ECGStandardization}. The torso geometry, heart placement, rotation convention, electrode locations, and lead definitions are fixed across all experiments.

This shared $V_m \to \phi_e \to u_T \to \text{ECG}$ pipeline ensures that when the two backends produce different ECGs, the differences can be attributed to the transmembrane dynamics. Each simulation also produces a diagnostics package--activation and repolarization summaries, quality metrics, and metadata for reproducible sampling and filtering--to support the diagnostics-aware curation described in Section~\ref{sec:curation}.

% ----------------------------------------------------------------------
\subsection{Activation-time theory: Bellman certificate}
\label{sec:theory}
% ----------------------------------------------------------------------

This section establishes a computable certificate for activation-time consistency on graphs. We show that the graph Eikonal activation time admits a Bellman fixed-point characterization, and derive a comparison principle that bounds the activation-time error using a computable residual. Proofs are given in the Supplementary Material (Section~S1).

\subsubsection{Definitions}

Let $G=(V,E)$ be a finite undirected graph with source set $S\subset V$ and $N:=|V|$. For each edge $\{i,j\}\in E$, let $d_{ij}>0$ be its geometric length and $v(i)>0$ a node-wise conduction speed. Define symmetric travel-time weights
\begin{equation}\label{eq:weight}
  w_{ij}
  \;:=\;
  \frac{d_{ij}}{\max\{v(i),\,v(j)\}},
\end{equation}
and the multi-source shortest-path activation time
\begin{equation}\label{eq:eikonal_time}
  T(i)
  \;:=\;
  \min_{s\in S}\;\min_{\gamma:\,s\to i}
  \sum_{k} w_{v_{k-1}v_k},
  \qquad
  T(s)=0
  \;\;\text{for } s\in S.
\end{equation}
On a graph, this activation time can be computed exactly by Dijkstra's algorithm. We retain the term ``Eikonal'' to emphasize the connection to wavefront propagation in cardiac electrophysiology~\cite{Keener1991EikonalCurvature, Sethian1996FastMarching, vanDam2009ActivationRecovery}, rather than the algorithmic implementation.

We introduce the Bellman operator $F:X\to X$ on the space $X:=\{x:V\to\mathbb{R}\cup\{\infty\}: x(s)=0\;\forall s\in S\}$:
\begin{equation}\label{eq:bellman_op}
  (Fx)(i)
  \;:=\;
  \begin{cases}
    0, & i\in S,\\[2pt]
    \displaystyle\min_{j\sim i}\bigl(x(j)+w_{ji}\bigr), & i\notin S,
  \end{cases}
\end{equation}
and the Bellman residual
\begin{equation}\label{eq:bellman_residual}
  \delta_{\mathrm{Bell}}(x) \;:=\; \|x - Fx\|_\infty,
\end{equation}
which measures how far a candidate time field~$x$ deviates from the Bellman fixed-point equation.

\subsubsection{Main results}

\begin{theorem}[Bellman fixed point]
\label{thm:bellman_fp}
The shortest-path activation time~$T$ is the unique element of~$X$ satisfying $T=FT$.
\end{theorem}

This establishes that the graph Eikonal activation problem is exactly a Bellman fixed-point problem, consistent with classical dynamic programming formulations~\cite{Bellman1958Routing, Dijkstra1959Graphs}.

\begin{theorem}[Comparison principle]
\label{thm:comparison}
Let $x\in X$ satisfy $\delta_{\mathrm{Bell}}(x)\le\delta$, and assume the greedy predecessor graph of~$x$ is acyclic. Then
\begin{equation}\label{eq:comparison}
  \|x - T\|_\infty \;\le\; m_{\max}\cdot\delta,
\end{equation}
where $m_{\max}$ is the maximum greedy depth of~$x$.
\end{theorem}

% The greedy predecessor graph is formed by directed edges $(\pi_x(i)\to i)$, where $\pi_x(i)\in\argmin_{j\sim i}(x(j)+w_{ji})$. Its acyclicity can be verified algorithmically in linear time. The greedy depth~$m_{\max}$ is an effective hop-diameter of the time ordering and is typically much smaller than the worst-case bound~$N-1$; in our experiments on random geometric graphs, the tightening factor reached 37--54$\times$ (Section~\ref{sec:results_certificate}).

Theorem~\ref{thm:comparison} is the core certificate: if a candidate activation field has small Bellman residual, it is uniformly close to the Eikonal time. This applies to any candidate field, whether produced by a reaction--diffusion solver, a reaction--eikonal surrogate, or any other method.

\subsubsection{Affine calibration}

In practice, simulator-produced activation times~$\tau$ are often related to the Eikonal time not by $\tau\approx T$ but by $\tau\approx\alpha T+\beta$, where $\alpha>0$ captures a global time-scale mismatch and~$\beta$ a source-time offset. To handle this, we define the affine-calibrated Bellman residual
\begin{equation}\label{eq:affine_residual}
  \delta_{\mathrm{affBell}}(\tau;\alpha,\beta)
  \;:=\;
  \bigl\|(\tau-\beta) - F_{\alpha w}(\tau-\beta)\bigr\|_\infty,
\end{equation}
where $F_{\alpha w}$ is the Bellman operator with scaled weights~$\alpha w_{ij}$. An affine comparison principle analogous to Theorem~\ref{thm:comparison} holds: if $\delta_{\mathrm{affBell}}$ is small, then $\|\tau-(\alpha T+\beta)\|_\infty$ is controlled (Supplementary Material, Section~S1). The calibration parameters $(\alpha,\beta)$ are estimated by least-squares regression of~$\tau$ against~$T$. In our heart-graph experiments, affine calibration improved the residual and yielded $R^2 = 0.967$ (Section~\ref{sec:results_certificate}).

% ----------------------------------------------------------------------
\subsection{Two solver backends: \textrm{ET} and \textrm{RE}}
\label{sec:solvers}
% ----------------------------------------------------------------------

We consider two forward models on the unified heart graph, sharing the same forward pipeline. They differ only in how they generate the transmembrane potential~$V_m$.

\subsubsection{ET: Eikonal-template baseline}

The Eikonal-template model~(\textrm{ET}) is a fast baseline built on a shortest-path activation field~$T$, computed using a Dijkstra-type method~\cite{Dijkstra1959Graphs, vanDam2009ActivationRecovery}. At time~$t$, the transmembrane signal at node~$i$ is represented by a prescribed action-potential template shifted in time by the local activation time~$T_i$:
\begin{equation}\label{eq:et}
  V_m(i,t) \;=\; A\!\bigl(t - T_i\bigr).
\end{equation}
The resulting field is passed to the common postprocessing chain to produce $\phi_e$ and the 12-lead ECG.

\textrm{ET} is fast and well suited for activation-oriented parameter sweeps. Its limitation is that recovery behavior is inherited from the chosen template: QT and T-wave features are not controlled independently of the activation field.

\subsubsection{RE: Adopted reaction--eikonal backend}

The reaction--eikonal model~(\textrm{RE}) keeps the same Eikonal activation backbone~$T$ but replaces the template shift with triggered ionic dynamics and spatial pseudo-diffusion, adapting the reaction--eikonal framework of Neic et~al.~\cite{Neic2017ReactionEikonal} to the graph setting with a simplified scalar coupling parameter. The evolution of the dimensionless voltage~$V_i$ at each node~$i$ is
\begin{equation}\label{eq:re}
  \frac{dV_i}{dt}
  \;=\;
  \underbrace{I_{\mathrm{ion}}^{\,\mathrm{tissue}(i)}(V_i,g_i)}_{\text{ionic model}}
  \;+\;
  \underbrace{B_i(t;\,T_i)}_{\text{trigger current}}
  \;-\;
  \underbrace{\kappa\!\sum_{j\sim i}\bar\sigma_{ij}\bigl(V_i-V_j\bigr)}_{\text{pseudo-diffusion}},
\end{equation}
with companion gating dynamics $dg_i/dt = h^{\mathrm{tissue}(i)}(V_i,g_i)$. Here $I_{\mathrm{ion}}$ and~$h$ are determined by a tissue-specific ionic model (Aliev--Panfilov~\cite{AlievPanfilov1996} or Bueno-Orovio~\cite{BuenoOrovio2008}), and $\bar\sigma_{ij}=\tfrac{1}{2}(\sigma_i+\sigma_j)$ is the averaged intracellular conductivity on edge~$(i,j)$.

The trigger current~$B_i(t;T_i)$ is a short pulse that initiates the local action potential at the time dictated by the Eikonal field~$T_i$, so the Eikonal backbone controls when each node depolarizes, while the ionic model controls the local waveform shape.

The pseudo-diffusion term couples neighboring nodes through a graph Laplacian restricted to myocardial edges, with scalar strength~$\kappa\ge 0$. When $\kappa=0$, the coupling vanishes and the model reduces to a set of independent triggered ODEs. When $\kappa>0$, neighboring repolarization fields are spatially coupled, smoothing recovery gradients and broadening the T-wave. Conduction-system and atrial edges are excluded so that the pseudo-diffusion does not distort the activation backbone.

The original reaction--eikonal model of Neic et~al.~\cite{Neic2017ReactionEikonal} operates on a finite-element mesh with a full monodomain diffusion operator. Our adaptation simplifies this to a graph Laplacian with a single scalar~$\kappa$, which makes the coupling strength directly interpretable as a morphology control parameter (Section~\ref{sec:knob_taxonomy}).

% Because \textrm{ET} and \textrm{RE} share the same activation backbone and the same forward pipeline, differences in ECG morphology can be attributed to the recovery layer rather than to changes in activation timing or postprocessing.

% ----------------------------------------------------------------------
\subsection{Knob taxonomy}
\label{sec:knob_taxonomy}
% ----------------------------------------------------------------------

The simulator parameters are organized into activation and recovery controls. We use this taxonomy as a testable knob-to-feature map: activation controls are expected to affect activation and QRS-related features, whereas recovery controls are expected to affect QT/T-wave features. The atlas and sensitivity experiments then test whether this separation is preserved when the recovery-aware \textrm{RE} backend is added.

\subsubsection{Activation controls}

\begin{equation}\label{eq:theta_act}
  \Theta_{\mathrm{act}}
  \;=\;
  \bigl\{\,
    \sigma_{\mathrm{purk,L}},\;
    \sigma_{\mathrm{purk,R}},\;
    \sigma_{\mathrm{AV}},\;
    \sigma_{\mathrm{LA/RA}},\;
    \sigma_{\mathrm{annulus}}
  \bigr\}.
\end{equation}

The activation controls modify conduction pathways on the shared heart graph. Left and right Purkinje conductivities ($\sigma_{\mathrm{purk,L}}$, $\sigma_{\mathrm{purk,R}}$) primarily affect ventricular activation timing, QRS morphology, and bundle-branch-block-like responses~\cite{Surawicz2009IVCD}. AV conductivity ($\sigma_{\mathrm{AV}}$) modulates atrioventricular delay, while the atrial transmural conductivity family ($\sigma_{\mathrm{LA/RA}}$) affects P-wave morphology. The annulus leak parameter ($\sigma_{\mathrm{annulus}}$) controls residual conduction across the fibrous ring~\cite{Saremi2017FibrousSkeleton}.

\subsubsection{Recovery controls}

\begin{equation}\label{eq:theta_rep}
  \Theta_{\mathrm{rep}}
  \;=\;
  \bigl\{\,
    \varepsilon_{0,\mathrm{endo}},\;
    \varepsilon_{0,\mathrm{epi}},\;
    \kappa
  \bigr\}.
\end{equation}

The recovery controls modify repolarization after the activation backbone has triggered local depolarization. The endo-epicardial APD gradient, controlled by the relative values of $\varepsilon_{0,\mathrm{endo}}$ and $\varepsilon_{0,\mathrm{epi}}$, is treated as the primary determinant of the admissible T-wave morphology window and QT/T behavior~\cite{Roden2008KeepQT,
Antzelevitch2007TpeakTend}. A global multiplicative scaling of $\varepsilon_0$ is used as an experimental perturbation to shift QT/T timing and amplitude, but is not treated as a standalone morphology-quality knob. The pseudo-diffusion strength~$\kappa$ (\textrm{RE} only) is interpreted as a secondary smoothing and modulation control: it reduces local repolarization roughness with little QTc displacement, rather than serving as the primary determinant of T-wave polarity or QT duration.

This activation-recovery separation is empirical rather than absolute. Section~\ref{sec:results_sensitivity} therefore quantifies normalized knob-feature coupling and a block disentanglement index for the activation and recovery controls summarized here.

% ----------------------------------------------------------------------
\subsection{Diagnostics-aware curation pipeline}
\label{sec:curation}
% ----------------------------------------------------------------------

Random parameter sampling can produce ECGs that are numerically stable but physiologically implausible. Each simulation is therefore evaluated by a two-stage diagnostics pipeline: fixed metric computation followed by experiment-specific acceptance. This separation allows the same diagnostics infrastructure to support broad recovery screening, recovery-gradient analysis, sensitivity analysis, and final multi-lead morphology curation.

\subsubsection{Stage 1: Metric computation}

The diagnostics vector contains four metric classes: activation-field diagnostics, ECG morphology and interval metrics, recovery scores, and morphology-coverage features.

\paragraph{Activation-field diagnostics}
Activation-field diagnostics test whether the simulated activation wavefront follows the intended conduction order~\cite{Durrer1970TotalExcitation}. These diagnostics include sequence reversal checks along the expected $\text{SA}\to\text{atrial}\to\text{AV}\to\text{His}\to\text{Purkinje}\to\text{ventricular}$ pathway, endo-epicardial ordering violations, large unexplained activation gaps, and causal predecessor consistency. Samples with gross activation-order failures are eligible for unconditional rejection under Stage~2.

\paragraph{ECG morphology and interval metrics}
ECG morphology and interval metrics are extracted from the simulated 12-lead ECG using a common detector~\cite{Kligfield2007ECGStandardization, Rautaharju2009STTQT}. Because the simulator provides latent activation and recovery fields in addition to the observable ECG trace, interval metrics are computed from a hybrid diagnostic representation: QRS timing is anchored by ventricular activation times and refined against ECG waveform landmarks, PR-related measures use atrial-to-ventricular activation timing, and QT/T-wave measures combine recovery-time summaries with lead-wise T-wave landmarks. This gives the curation pipeline access to mechanism-aware intervals that are not directly available to purely waveform-generative models; detailed metric definitions are provided in the Supplementary Material. Waveform metrics include lead-wise P-wave amplitude, QRS amplitude, T-wave amplitude, selected polarity checks, and gross morphology-failure flags such as flatline-like traces, global QRS inversion, missing QRS complexes, and implausible wave ordering. Because timing contexts differ across display, atlas, and curation analyses, interval values are interpreted within each experiment rather than pooled as absolute clinical measurements.

\paragraph{Recovery plausibility and repolarization scores}
Recovery plausibility and repolarization scores are internal simulation-quality scores for repolarization behavior. The recovery plausibility score, denoted $s_{\mathrm{rep}}$, is the policy-level screening score; it summarizes recovery-gradient consistency, local repolarization smoothness, and T-wave morphology into a scalar criterion. A related repolarization score summarizes local recovery quality, especially smoothness of neighboring repolarization times, and is reported as a mechanism-oriented diagnostic rather than as the primary acceptance score. These scores are used to rank and filter parameter samples, but they are not intended to replace clinical repolarization biomarkers.

\paragraph{Feature-space coverage}
Feature-space coverage is used in the final curation experiment to quantify morphology diversity among accepted samples. Selected ECG features are discretized into bins, and coverage is defined as the fraction of occupied bins in the admissible feature space. Per-model coverage compares \textrm{ET} and \textrm{RE} under the same sampling budget, whereas pooled coverage summarizes the combined accepted set.

\subsubsection{Stage 2: Policy-based acceptance}

The computed metrics are compared with acceptance rules that depend on the experimental endpoint. Shared hard filters, such as gross activation-order failure, global QRS inversion, missing QRS complexes, and flatline-like traces, lead to unconditional rejection. Softer constraints, including $s_{\mathrm{rep}}$ thresholds, timing bounds, T-wave polarity requirements, and feature-balance criteria, are policy-specific. E2 uses a broad recovery-score screen for throughput estimation, E4 uses recovery-specific screens, E5 is not acceptance-driven, and E6 uses the final balanced multi-lead curation policy. The experiment-level assignments are summarized in Table~\ref{tab:experiment_overview}, and the diagnostic quantities used by these policies are defined in the Supplementary Material.

The Bellman residual from Section~\ref{sec:theory} is related to this pipeline through the activation-field diagnostics: the sequence, gap, and predecessor checks are practical implementations of the activation-consistency idea formalized by the certificate.

% ----------------------------------------------------------------------
\subsection{Experimental design overview}
\label{sec:exp_overview}
% ----------------------------------------------------------------------

All experiments share the same baseline configuration: a unified heart graph with tissue-specific Aliev--Panfilov parameters, separate intracellular conductivities per tissue type, fixed torso geometry and lead placement, and common solver settings ($\Delta t = 0.129$\,ms).

The experiments are organized in two groups:

\paragraph{Certificate validation (Section~\ref{sec:results_certificate}).}
We test the Bellman residual certificate on toy graphs (chain, grid, random geometric) and the heart graph, and verify stability under time-step variation.

\paragraph{ECG synthesis validation (Sections~\ref{sec:results_activation}--\ref{sec:results_curation}).}
We test the \textrm{ET} and \textrm{RE} backends through the six experiments summarized in Table~\ref{tab:experiment_overview}.

\begin{table}[t]
  \centering
  \caption{Overview of ECG synthesis experiments.}
  \label{tab:experiment_overview}
  \footnotesize
  \begin{tabular}{@{}p{1.1cm}p{3.0cm}p{3.0cm}p{4.8cm}@{}}
    \toprule
    \textbf{Exp.} & \textbf{Purpose} & \textbf{Varied factors} & \textbf{Primary endpoint} \\
    \midrule
    E1 & Baseline ECG morphology & None; tuned baseline & Stable P-QRS-T morphology in \textrm{ET} and \textrm{RE} \\
    E2 & Throughput screening & Random parameter samples & Acceptance under a broad recovery-score policy \\
    E3 & Activation atlas & AV, annulus, and atrial-transmural parameters & Relative activation knob-response trends \\
    E4 & Recovery atlas & $\epsilon_0$, $\kappa$, and endo-epi APD gradient & QT/T response and admissible recovery-gradient window \\
    E5 & Sensitivity analysis & Local perturbations of activation and recovery knobs & Knob-feature coupling and activation/recovery disentanglement \\
    E6 & Final morphology curation & Random parameter samples & Balanced multi-lead morphology yield and feature-space coverage \\
    \bottomrule
  \end{tabular}
\end{table}

Individual experiments vary only the parameters stated in their respective descriptions; all other values remain at the baseline defaults.

% ----------------------------------------------------------------------
\subsection{Downstream low-label P/QRS/T delineation}
\label{sec:methods_downstream}
% ----------------------------------------------------------------------

To test whether the mechanistically generated ECGs provide a useful initialization signal beyond internal curation metrics, we performed a low-label P/QRS/T delineation evaluation. This experiment is treated as an application-oriented utility assessment rather than as a full clinical delineation benchmark.

Synthetic pretraining data were generated from curated single-beat ECGs by varying heart rate, PR timing, QRS duration, QT timing, waveform amplitude, and lead morphology, then concatenating beats into 12-s clean multi-lead traces. The learning target was a common timeline segmentation mask with four classes: background, P wave, QRS complex, and T wave. Evaluation was performed on Lead~I, Lead~II, and V1.

We compared three training conditions under identical real-data splits:
(1) a real-only training baseline from scratch;
(2) random synthetic pretraining followed by real-data fine-tuning; and
(3) parameter-informed synthetic pretraining followed by real-data fine-tuning.
Rather than sampling simulation knobs purely at random, the parameter-informed setting uses a curated set of 12 synthetic parameter templates selected to span biophysically motivated normal sinus rhythm (NSR), left bundle branch block (LBBB), and right bundle branch block (RBBB) morphologies (Figure~\ref{fig:templates_12}). This process generated 2,000 synthetic multi-beat traces from the representative templates, serving as targeted parameter-informed augmentation without fitting clinical waveforms during training. The conduction-morphology variation in these templates was intended to expose the segmentation model to controlled P/QRS/T timing and shape diversity during pretraining.

The available real training pool was subsampled at 1\% (56 records), 2\% (111 records), 5\% (277 records), and 10\% (555 records), with four subset seeds per fraction. Model selection used the validation split, and all reported metrics were computed on the held-out PTB-XL+ test split. The primary metrics were P-wave, QRS-complex, and T-wave IoU.

\begin{figure}[t]
  \centering
  \includegraphics[width=0.98\linewidth]{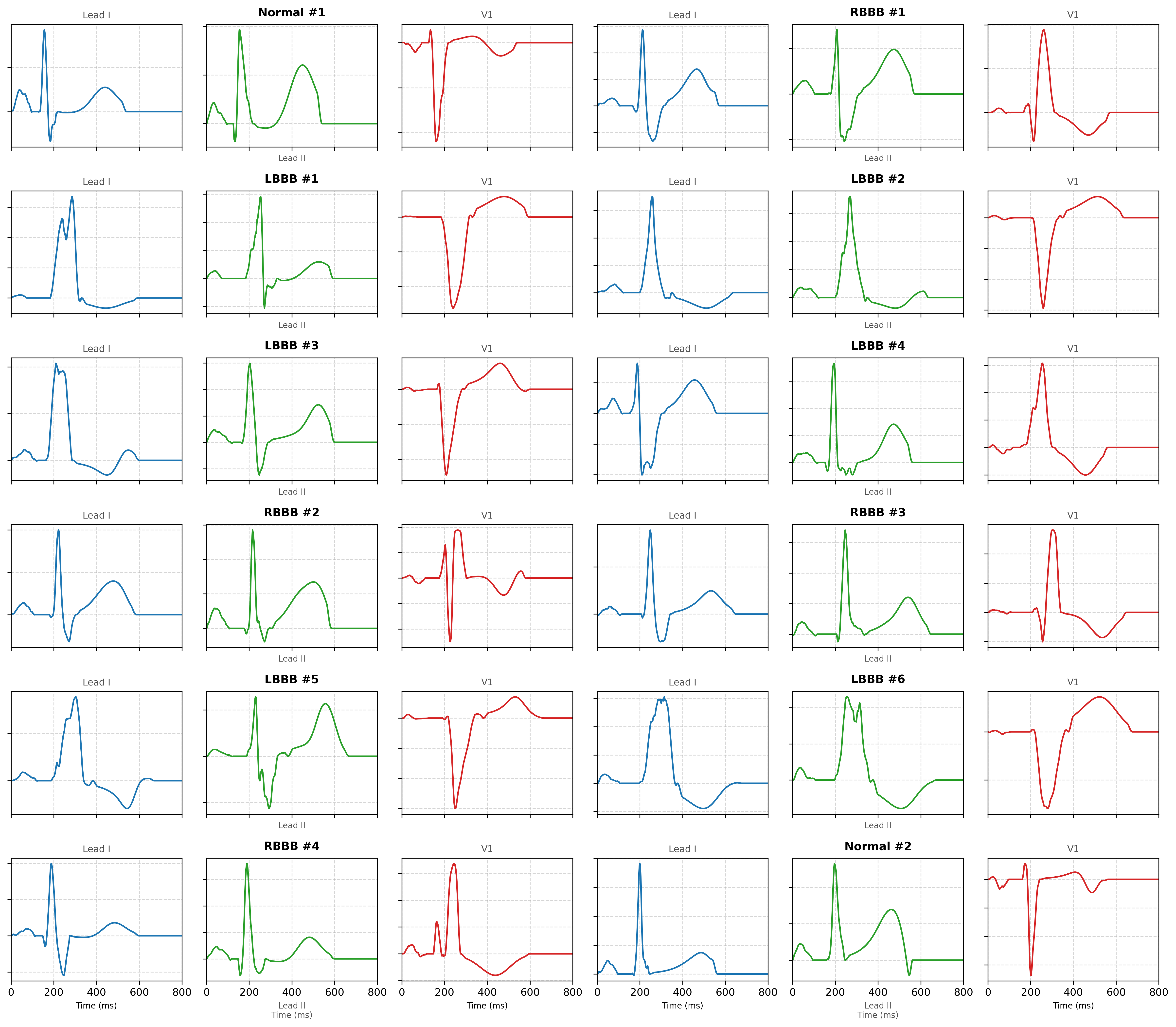}
  \caption{Synthetic parameter templates used for parameter-informed pretraining. The 12 templates span biophysically motivated NSR, LBBB, and RBBB-like conduction morphologies, shown here in Leads~I, II, and V1. These waveforms are intended to provide controlled P/QRS/T timing and shape variation for pretraining, rather than diagnostic exemplars of clinical bundle-branch block.}
  \label{fig:templates_12}
\end{figure}

% ======================================================================
\section{Results}
\label{sec:results}
% ======================================================================

This section evaluates the graph-based ECG synthesis framework in two stages. First, we test whether activation times extracted from the reaction-eikonal backend remain consistent with the certified eikonal activation backbone. Second, using a single tuned baseline for all ECG synthesis experiments, we evaluate whether the same graph and ECG forward model support stable baseline morphology, interpretable activation and recovery parameter responses, and diagnostics-aware synthetic ECG curation. Score definitions, morphology criteria, and policy assignments are described in Methods; this section reports the resulting behavior under those policies.

% ----------------------------------------------------------------------
\subsection{Bellman certificate validation}
\label{sec:results_certificate}
% ----------------------------------------------------------------------

\paragraph{Toy and heart-graph certificate checks}
We first evaluated the Bellman residual certificate on controlled graph examples. In 50 chain-graph perturbation scenarios, generated by modifying edge geometry or introducing junction delays, the Bellman residual remained bounded by the edge mismatch in every case. No predecessor cycles were detected. These tests verify the basic comparison principle under controlled conditions.

On the cardiac graph, the reference eikonal activation time was computed on the tuned graph with 1321 nodes and 4546 edges. All nodes were reachable from the source set, and the exact eikonal solution satisfied the Bellman fixed-point equation with zero residual. We then compared this eikonal backbone with activation times extracted from the reaction-eikonal simulation. To avoid spurious predecessor cycles caused by near-simultaneous phase-0 detections, activation times were extracted using a phase-0 start marker with a dV/dt threshold of 5 and evaluated with a causal predecessor filter.

With the causal predecessor map, the extracted activation field remained close to the eikonal backbone: the Bellman residual was $0.9258$ ms, the sup-norm activation error was $0.9400$ ms, and the linear fit achieved $R^2=0.99876$. The greedy predecessor depth was 22, giving a tightened certificate bound of $20.37$ ms, and the causal predecessor graph contained no cycles.

\paragraph{Time-step scaling}
Time-step scaling confirmed that the certificate behavior was stable over the tested range. For $dt=\{0.025,0.05,0.1,0.2\}$ ms, the activation error ranged from $0.975$ to $1.214$ ms, and $R^2$ remained between $0.99831$ and $0.99860$ (Table~\ref{tab:dt_scaling_new}). No causal predecessor cycles occurred at any tested time step, supporting stability of the extracted activation ordering.

\begin{table}[t]
  \centering
  \caption{Causal Bellman certificate diagnostics under time-step variation.}
  \label{tab:dt_scaling_new}
  \small
  \begin{tabular}{@{}lcccc@{}}
    \toprule
    $dt$ (ms) & $\delta_{\mathrm{Bell}}$ & $e_\infty$ (ms) & $R^2$ & cycles \\
    \midrule
    0.025 & 0.9608 & 0.9749 & 0.99860 & 0 \\
    0.050 & 0.9358 & 0.9999 & 0.99860 & 0 \\
    0.100 & 1.0163 & 1.0998 & 0.99835 & 0 \\
    0.200 & 1.0358 & 1.2140 & 0.99831 & 0 \\
    \bottomrule
  \end{tabular}
\end{table}

Together, these results support the use of the eikonal activation field as a certificate-checked activation backbone for the ECG synthesis experiments. The Bellman residual does not provide a tight numerical error bound in this setting, but it acts as a computable consistency certificate for the extracted activation ordering.

% ----------------------------------------------------------------------
\subsection{Baseline morphology and activation atlas}
\label{sec:results_activation}
% ----------------------------------------------------------------------

\paragraph{Baseline ECG morphology}
All ECG synthesis experiments were performed using a single tuned baseline selected to produce stable P-QRS-T morphology in Lead~I, Lead~II, and V1. The baseline was accepted for both ET and RE, supporting its use as a common morphology anchor for the atlas, sensitivity, and curation experiments (Table~\ref{tab:baseline_sanity}). The corresponding baseline ECG overlay is shown in Figure~\ref{fig:baseline_ecg_overlay}.

\begin{figure}[t]
  \centering
  \includegraphics[width=0.98\linewidth]{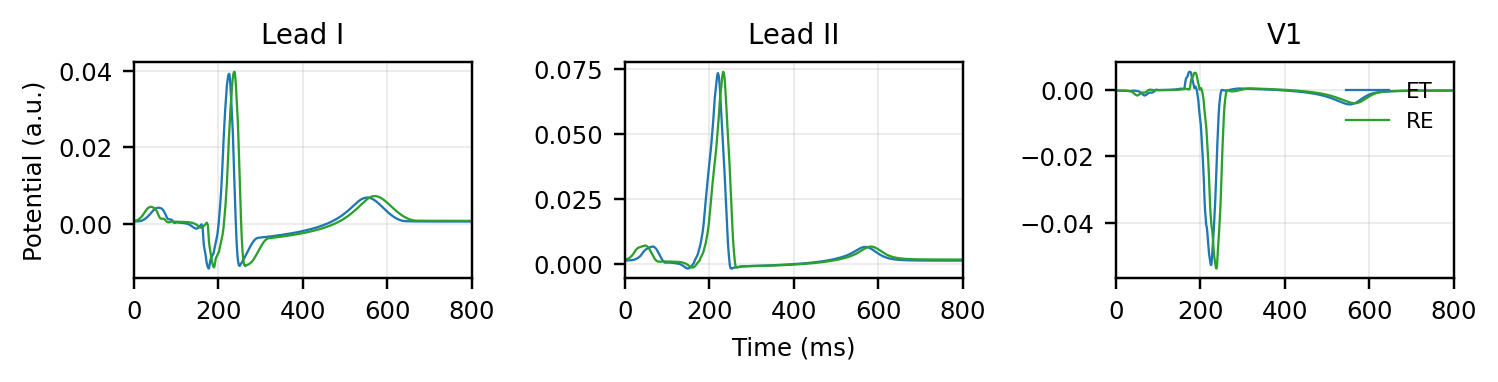}
  \caption{Baseline ECG overlay for ET and RE.}
  \label{fig:baseline_ecg_overlay}
\end{figure}

\begin{table}[t]
  \centering
  \caption{Baseline ECG morphology summary.}
  \label{tab:baseline_sanity}
  \begingroup
  \setlength{\tabcolsep}{4pt}
  \footnotesize
  \begin{tabular}{@{}lccccc@{}}
    \toprule
    Model & PR & QRS & QTc & P amp. (II) & T amp. (II) \\
    \midrule
    ET & 148.06 & 82.78 & 352.93 & 0.00375 & 0.00594 \\
    RE & 142.02 & 82.78 & 354.36 & 0.00520 & 0.00605 \\
    \bottomrule
  \end{tabular}
  \endgroup
\end{table}

\paragraph{Activation atlas}
The activation atlas varied one activation-side knob at a time around the common baseline: AV conductivity and annulus leak were used to probe PR/atrioventricular timing responses, whereas the left-atrial transmural parameter was used to probe P-wave morphology with minimal ventricular timing change. Across these sweeps, RE tracked ET closely in display-calibrated PR timing, QRS range, P-wave amplitude, and QRS amplitude (Table~\ref{tab:activation_atlas_summary}). This supports the interpretation that adding pseudo-diffusion reaction--eikonal recovery dynamics preserves the activation knob-response behavior established by the eikonal baseline.

\begin{table}[t]
  \centering
  \caption{Activation-atlas summary.}
  \label{tab:activation_atlas_summary}
  \begingroup
  \setlength{\tabcolsep}{4pt}
  \footnotesize
  \begin{tabular}{@{}llcccc@{}}
    \toprule
    Knob & Model & PR & QRS & P amp. & QRS amp. \\
    \midrule
    AV cond. & ET & 174.84--294.37 & 82.78--82.78 & 0.0037 & 0.0729 \\
    AV cond. & RE & 168.64--288.17 & 82.78--83.16 & 0.0052 & 0.0728 \\
    Annulus leak & ET & 145.12--386.46 & 82.78--82.78 & 0.0037 & 0.0725 \\
    Annulus leak & RE & 138.92--380.43 & 82.78--82.78 & 0.0052 & 0.0724 \\
    LA transmural & ET & 148.06--148.28 & 82.78--82.78 & 0.0037 & 0.0724 \\
    LA transmural & RE & 141.80--142.02 & 82.78--82.78 & 0.0050 & 0.0729 \\
    \bottomrule
  \end{tabular}
  \endgroup
\end{table}

% ----------------------------------------------------------------------
\subsection{Recovery atlas}
\label{sec:results_recovery}
% ----------------------------------------------------------------------

The recovery atlas varied recovery-side parameters one at a time around the same baseline: diffusion strength $\kappa$ tested local repolarization smoothing, global $\varepsilon_0$ scaling tested QT/T timing shifts, and the endo-epicardial APD gradient tested the admissible T-wave morphology window. A mechanistic recovery screen evaluated how these parameters shape QT/T behavior, whereas a stricter endo-epi gradient sweep identified the admissible morphology window. We denote the recovery plausibility score by $s_{\mathrm{rep}}$ throughout. Together, these experiments show that the endo-epicardial APD gradient primarily determines the acceptable T-wave morphology window, while $\kappa$ provides secondary smoothing of repolarization without substantially shifting QT timing.

In the mechanistic sweep, increasing $\kappa$ improved $s_{\mathrm{rep}}$ from 42.25 at $\kappa=0$ to 47.13 at $\kappa=0.125$, while QTc remained approximately 354 ms. The 95th percentile neighbor difference in repolarization time decreased from 20.13 ms to 16.71 ms, consistent with a smoothing effect. Global $\varepsilon_0$ scaling produced the expected timing response: lower scale prolonged QTc and increased Lead II T amplitude, whereas higher scale shortened QTc and reduced T amplitude. However, global scaling alone did not provide a strong
morphology-quality control.

The strict gradient sweep accepted 4/9 ET samples and 5/9 RE samples. The best $s_{\mathrm{rep}}$ was 92.61 for ET and 93.42 for RE (Table~\ref{tab:recovery_atlas_summary}). Physiologically interpretable positive-T cases were concentrated at gradients $+0.004$, $+0.008$, $+0.013$, and $+0.018$. The RE case at $-0.007$ passed the strict recovery policy but was excluded from the positive-T interpretation because that policy did not explicitly require positive Lead II T polarity. The recovery-atlas summary is reported in Table~\ref{tab:recovery_atlas_summary}.

\begin{table}[t]
  \centering
  \caption{Recovery-atlas summary.}
  \label{tab:recovery_atlas_summary}
  \begingroup
  \setlength{\tabcolsep}{4pt}
  \footnotesize
  \begin{tabular}{@{}lccc@{}}
    \toprule
    Component & ET & RE & Metric \\
    \midrule
    $\kappa$ & -- & 5/9 & $s_{\mathrm{rep}}$ 42.25--47.13 \\
    Global $\varepsilon_0$ & 0/7 & 1/7 & QTc 336--387 ms \\
    Endo-epi gradient & 4/9 & 5/9 & $+0.004$--$+0.018$ \\
    Best $s_{\mathrm{rep}}$ & 92.61 & 93.42 & RE higher \\
    \bottomrule
  \end{tabular}
  \endgroup
\end{table}

% ----------------------------------------------------------------------
\subsection{Sensitivity and disentanglement}
\label{sec:results_sensitivity}
% ----------------------------------------------------------------------

We quantified local knob-feature coupling using normalized sensitivity matrices and a block disentanglement index. Both models showed strong separation between activation-sensitive controls and recovery-sensitive features. The disentanglement index was 0.9987 for ET and 0.9985 for RE, indicating that adding pseudo-diffusion reaction--eikonal recovery dynamics did not collapse the activation/recovery interpretability of the parameterization.

The normalized sensitivity maps showed that Purkinje parameters dominated activation-side timing features, whereas $\varepsilon_0$ parameters dominated QT, T-wave amplitude, and $T_{\mathrm{peak}}$--$T_{\mathrm{end}}$. The added $\kappa$ parameter in RE had a small and localized influence on recovery features. This weak $\kappa$ response is consistent with the feature definition: the sensitivity map summarizes global ECG intervals and amplitudes, whereas $\kappa$ mainly affects local repolarization smoothness, as quantified in the recovery atlas (Figure~\ref{fig:sensitivity_maps}).

\begin{figure}[t]
  \centering
  \begin{subfigure}[t]{0.48\linewidth}
    \centering
    \includegraphics[width=\linewidth]{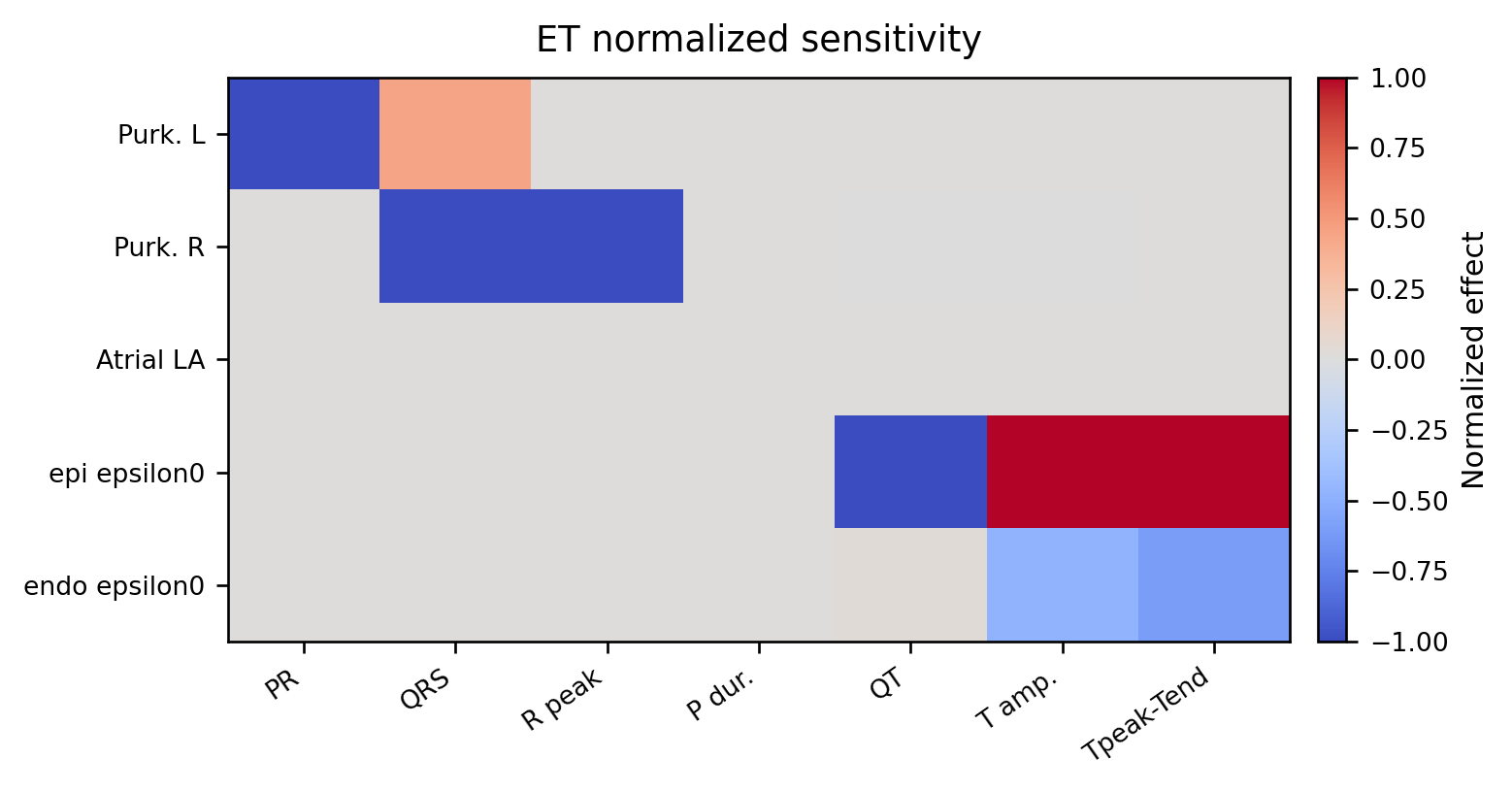}
    \caption{ET.}
  \end{subfigure}
  \hfill
  \begin{subfigure}[t]{0.48\linewidth}
    \centering
    \includegraphics[width=\linewidth]{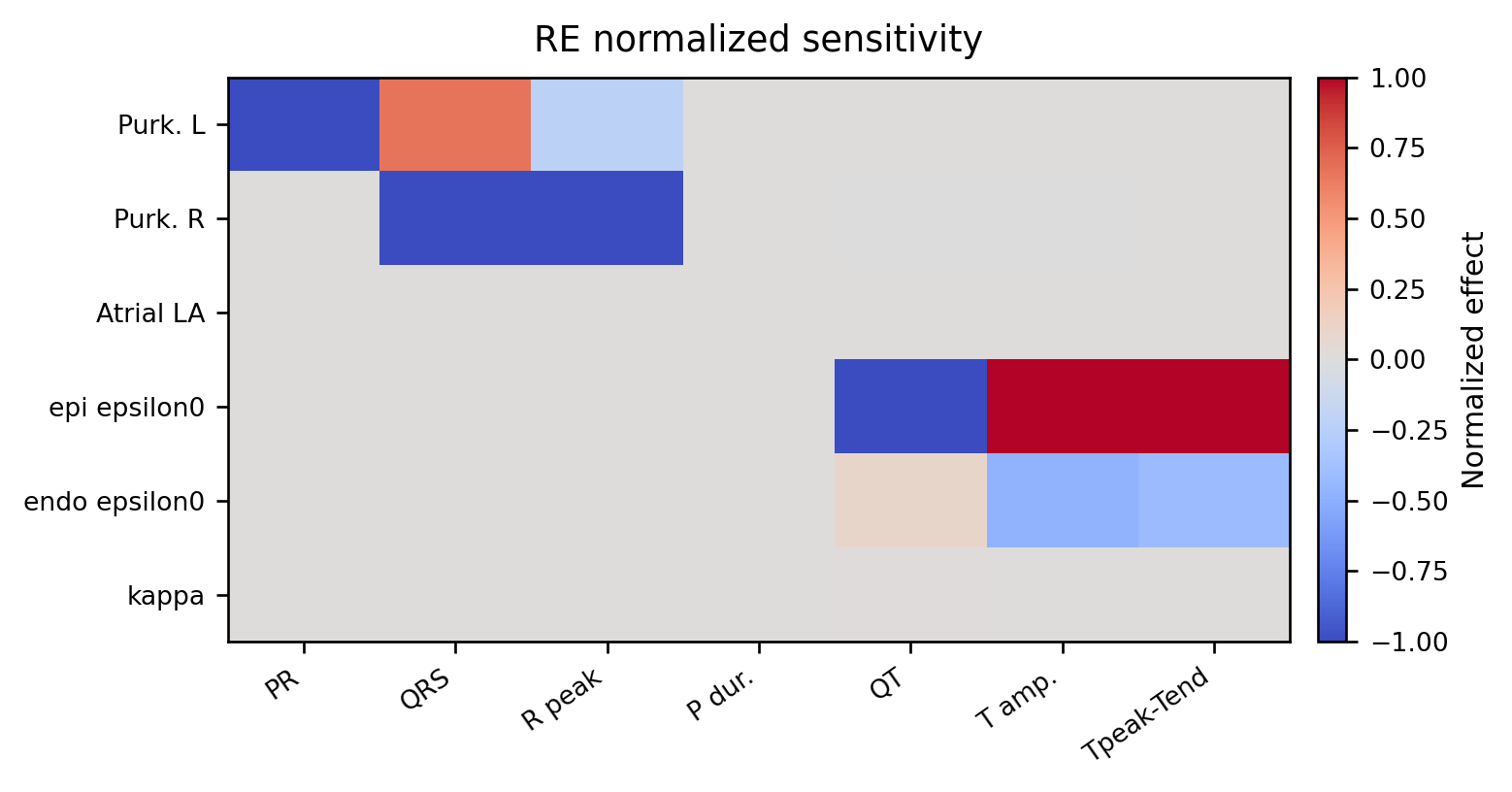}
    \caption{RE.}
  \end{subfigure}
  \caption{Normalized local knob-feature sensitivity maps. Purkinje controls dominate activation timing, $\varepsilon_0$ controls QT/T-wave features, and the weak RE $\kappa$ row reflects its mainly local smoothing effect.}
  \label{fig:sensitivity_maps}
\end{figure}

% ----------------------------------------------------------------------
\subsection{Diagnostics-aware throughput and curation}
\label{sec:results_curation}
% ----------------------------------------------------------------------

Finally, we evaluated the framework as a diagnostics-aware synthetic ECG generator. E2 was used as a recovery-score throughput screen with a higher recovery-plausibility threshold, whereas E6 used the balanced multi-lead morphology curation policy with an explicit Lead~II T-polarity gate and a lower recovery-score threshold. This distinction is important: E2 estimates throughput under a recovery-score screen, while E6 applies the curation policy used for morphology coverage claims.

In E2, ET accepted 200/1000 samples (20.0\%), whereas RE accepted 265/1000 samples (26.5\%). The rejection profile was dominated by the balanced recovery score. The lower E2 acceptance rate relative to E6 therefore reflects the higher recovery-score threshold rather than a stricter morphology-curation policy. This supports a modest RE advantage under the recovery-score throughput screen, but E2 is not interpreted as final ECG morphology curation because it does not enforce positive Lead~II T polarity.

E6 provides the main curation result. Under the same 2000-sample random budget per model, ET accepted 578 samples (28.9\%) and RE accepted 658 samples (32.9\%). The balanced subset contained 538 ET samples and 600 RE samples. Per-model coverage increased from 0.09248 for ET to 0.09888 for RE, while accepted sample means remained similar across models (Table~\ref{tab:curation_policy_summary}). The dominant rejection modes were non-positive Lead II T waves and low balanced recovery score, either alone or in combination. Representative accepted and rejected examples are shown in Figure~\ref{fig:curation_examples}; these examples provide a qualitative check of the curation policy rather than evidence of clinical representativeness. The accepted-feature distributions are shown in Figure~\ref{fig:curation_summary}.

\begin{figure}[t]
  \centering
  \begin{subfigure}[t]{0.49\linewidth}
    \centering
    \includegraphics[width=\linewidth]{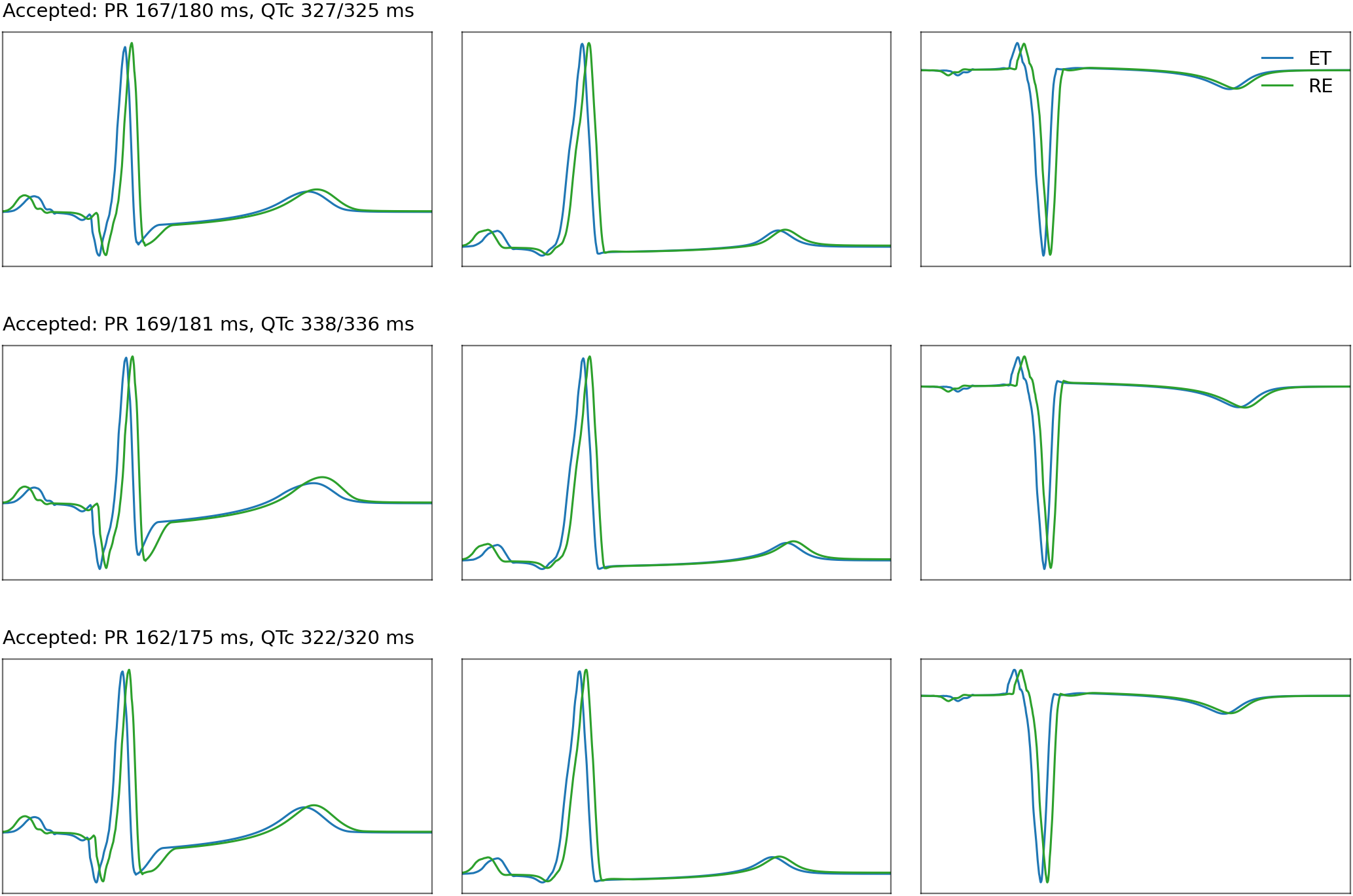}
    \caption{Accepted samples.}
  \end{subfigure}
  \hfill
  \begin{subfigure}[t]{0.49\linewidth}
    \centering
    \includegraphics[width=\linewidth]{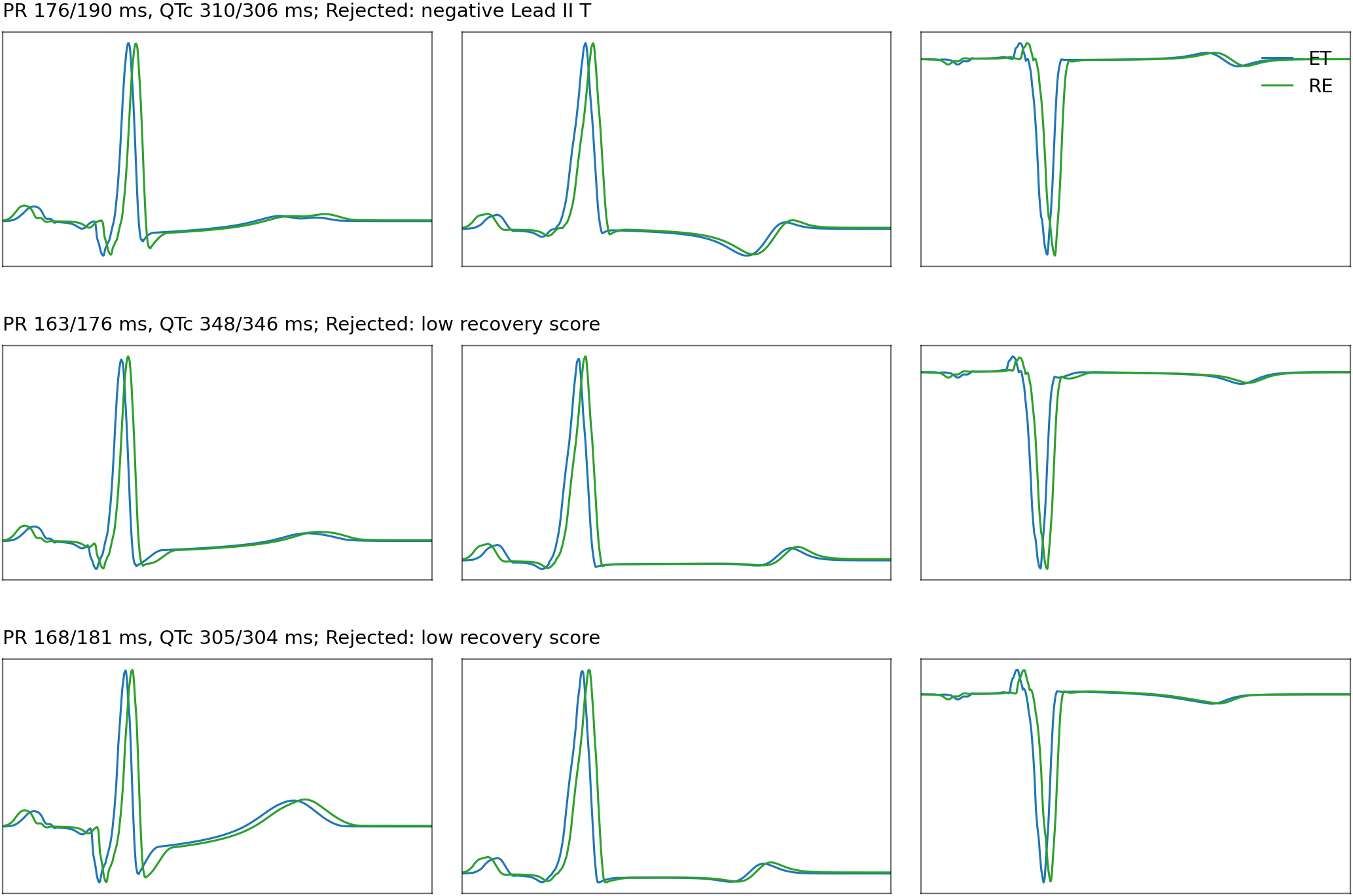}
    \caption{Rejected samples.}
  \end{subfigure}
  \caption{Representative accepted and rejected synthetic ECG samples from the final curation experiment. Each panel shows Lead~I, Lead~II, and V1 for three seeds with overlaid \textrm{ET} and \textrm{RE} traces. Rejected examples illustrate the dominant policy failures, including non-positive Lead~II T waves and low balanced recovery score.}
  \label{fig:curation_examples}
\end{figure}

\begin{figure}[t]
  \centering
  \includegraphics[width=0.88\linewidth]{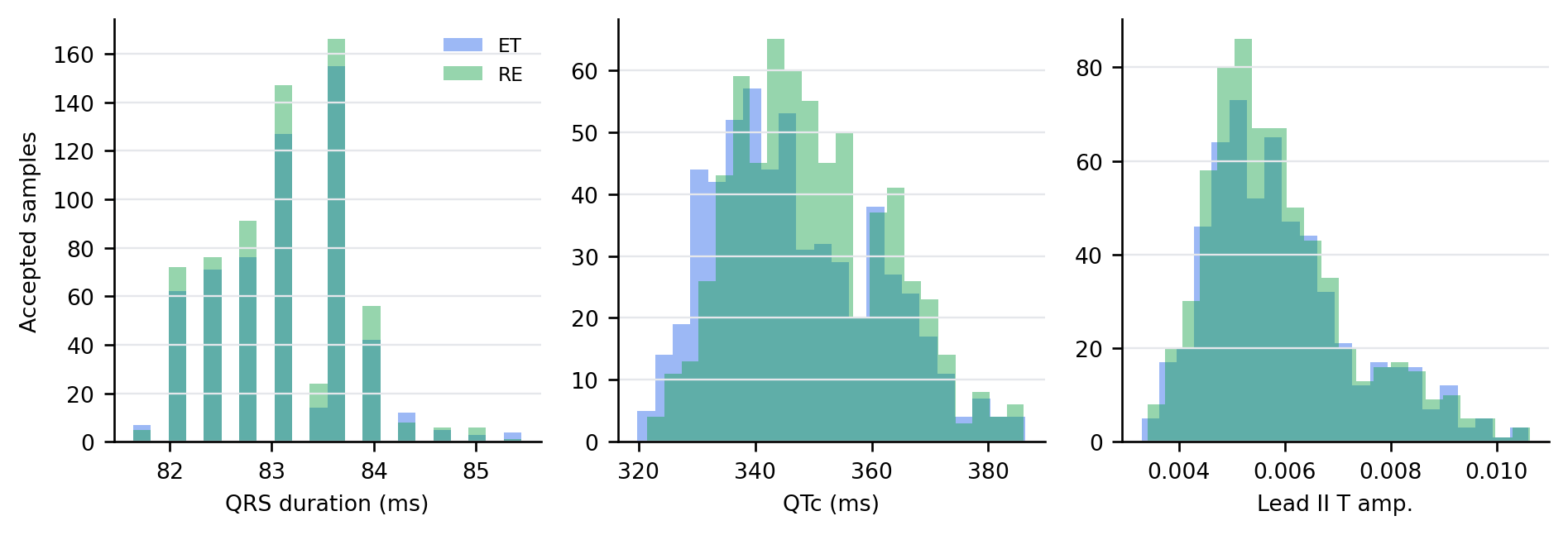}
  \caption{Accepted-feature distributions after diagnostics-aware curation under a fixed 2000-sample budget per model. Accepted ET and RE samples remain broadly comparable; yield and coverage are summarized in Table~\ref{tab:curation_policy_summary}.}
  \label{fig:curation_summary}
\end{figure}

\begin{table}[t]
  \centering
  \caption{Policy-specific throughput and curation outcomes under fixed sampling budgets.}
  \label{tab:curation_policy_summary}
  \small
  \begin{tabular}{@{}llccp{0.30\linewidth}@{}}
    \toprule
    Experiment & Model & Accepted / generated & Coverage & Main rejection mode \\
    \midrule
    E2 throughput & ET & 200/1000 (20.0\%) & -- & Recovery score \\
    E2 throughput & RE & 265/1000 (26.5\%) & -- & Recovery score \\
    E6 curation & ET & 578/2000 (28.9\%) & 0.09248 & T polarity / recovery score \\
    E6 curation & RE & 658/2000 (32.9\%) & 0.09888 & T polarity / recovery score \\
    \bottomrule
\end{tabular}
\end{table}

% ----------------------------------------------------------------------
\subsection{Downstream low-label P/QRS/T delineation}
\label{sec:results_downstream}
% ----------------------------------------------------------------------

As an application-oriented evaluation, we tested whether the curated synthetic ECGs provide a useful initialization signal for P/QRS/T delineation when only a small fraction of PTB-XL+ normal-sinus annotations is available. Across four subset seeds per fraction, we evaluated three conditions: a training-from-scratch baseline (Real-Only), random synthetic pretraining (Random Pretrain), and parameter-informed synthetic pretraining (Parameter-Informed).

In the extreme low-label regime (1\% real labels, corresponding to 56 clinical records), the real-only baseline exhibits high vulnerability to subset sampling noise, resulting in severe model collapse in one instance (a T-wave IoU drop to $0.3200$ under subset seed 3, leading to an overall seed-averaged T-wave IoU of $0.7478 \pm 0.2470$). No such collapse was observed in the tested pretrained runs, and the seed-wise T-wave IoU standard deviation decreased to $\le 0.003$.

Furthermore, parameter-informed pretraining provides small additional accuracy gains over random pretraining in the most label-scarce settings (Table~\ref{tab:downstream_pqrs}). At 1\% label fraction, parameter-informed pretraining achieves higher average IoU than random pretraining for P-wave ($0.8628$ vs. $0.8595$), QRS ($0.9059$ vs. $0.8990$), and T-wave ($0.8967$ vs. $0.8928$). At 2\% real labels, it continues to achieve the highest P IoU ($0.8681$) and QRS IoU ($0.9130$), whereas the real-only baseline achieves slightly higher T-wave IoU ($0.9022$). As real label availability increases to 5\% and 10\%, the performance differences among the three settings become marginal, which is expected as 277 to 555 clinical records provide sufficient data to train the U-Net from scratch without pretraining. Across most wave/fraction combinations, synthetic pretraining reduces seed-to-seed variance relative to real-only training, while the parameter-informed templates mainly provide a small accuracy advantage in the most label-scarce settings.

\begin{table}[t]
  \centering
  \caption{PTB-XL+ low-label P/QRS/T delineation performance. Values represent the mean $\pm$ standard deviation of the test IoU across four random subset seeds.}
  \label{tab:downstream_pqrs}
  \begingroup
  \setlength{\tabcolsep}{3pt}
  \footnotesize
  \begin{tabular}{@{}llccc@{}}
    \toprule
    Real labels & Wave & Real-Only Baseline & Random Pretrain & Parameter-Informed \\
    \midrule
    1\% 
    & P   & $0.8138 \pm 0.0578$ & $0.8595 \pm 0.0038$ & $\mathbf{0.8628 \pm 0.0030}$ \\
    & QRS & $0.8480 \pm 0.0807$ & $0.8990 \pm 0.0017$ & $\mathbf{0.9059 \pm 0.0018}$ \\
    & T   & $0.7478 \pm 0.2470$ & $0.8928 \pm 0.0005$ & $\mathbf{0.8967 \pm 0.0016}$ \\
    \midrule
    2\% 
    & P   & $0.8582 \pm 0.0036$ & $0.8652 \pm 0.0027$ & $\mathbf{0.8681 \pm 0.0014}$ \\
    & QRS & $0.9022 \pm 0.0034$ & $0.9051 \pm 0.0031$ & $\mathbf{0.9130 \pm 0.0009}$ \\
    & T   & $\mathbf{0.9022 \pm 0.0027}$ & $0.8978 \pm 0.0013$ & $0.9018 \pm 0.0006$ \\
    \midrule
    5\% 
    & P   & $0.8723 \pm 0.0038$ & $\mathbf{0.8733 \pm 0.0028}$ & $0.8716 \pm 0.0011$ \\
    & QRS & $0.9037 \pm 0.0069$ & $0.9096 \pm 0.0009$ & $\mathbf{0.9140 \pm 0.0009}$ \\
    & T   & $\mathbf{0.9027 \pm 0.0033}$ & $0.9003 \pm 0.0028$ & $0.8983 \pm 0.0029$ \\
    \midrule
    10\% 
    & P   & $0.8722 \pm 0.0032$ & $0.8729 \pm 0.0027$ & $\mathbf{0.8752 \pm 0.0002}$ \\
    & QRS & $0.9074 \pm 0.0026$ & $0.9098 \pm 0.0031$ & $\mathbf{0.9099 \pm 0.0015}$ \\
    & T   & $\mathbf{0.9101 \pm 0.0047}$ & $0.9065 \pm 0.0022$ & $0.9064 \pm 0.0025$ \\
    \bottomrule
  \end{tabular}
  \endgroup
\end{table}

These results suggest that mechanistically generated P/QRS/T structures provide a useful domain-specific initialization signal that stabilizes low-label segmentation models. In the most label-scarce settings, using biophysically motivated NSR, LBBB, and RBBB parameter templates provides a slightly more targeted structural initialization than purely random parameter sampling, while both synthetic pretraining strategies reduce the pronounced instability observed in the 1\% real-only baseline.

Taken together, these experiments show that RE preserves ET-like activation controllability while modestly improving diagnostically curated ECG morphology yield and coverage under balanced multi-lead curation. The improvement is not large enough to support a strong superiority claim, but it is consistent across accepted count, balanced subset size, and per-model coverage.

% ======================================================================
\section{Discussion}
\label{sec:discussion}
% ======================================================================

This study developed a diagnostics-aware framework for synthetic ECG generation on a unified cardiac graph. The main finding is that activation consistency, morphology control, and sample curation can be treated as coupled but separable components of the generation pipeline. The Bellman residual provides a computable activation-consistency check for graph-based eikonal activation fields, while the diagnostics pipeline converts generated ECGs into policy-dependent accepted or rejected samples. Together, these components support a reproducible, lightweight approach to quality-controlled synthetic ECG generation.

The activation-certificate results support the use of the eikonal field as a certificate-checked backbone for the ECG synthesis experiments. On the cardiac graph, activation times extracted from the \textrm{RE} backend remained close to the eikonal reference after causal predecessor filtering, with near-millisecond error and no predecessor cycles. The residual bound was not tight enough to be interpreted as a precise numerical error estimate, but it was useful as a consistency diagnostic for the extracted activation ordering.

The ECG synthesis experiments indicate that adding pseudo-diffusion reaction--eikonal recovery dynamics did not disrupt activation controllability. Across activation-atlas sweeps, \textrm{RE} tracked \textrm{ET} closely in relative knob-response trends. This supports the design choice of using a shared activation backbone and ECG forward pipeline, so that differences between \textrm{ET} and \textrm{RE} can be interpreted mainly as consequences of the recovery layer. The sensitivity analysis further supported this separation, with high activation/recovery disentanglement indices for both models.

The recovery experiments clarify the roles of the recovery controls. The endo-epicardial APD gradient was the main determinant of the admissible T-wave morphology window. Global $\varepsilon_0$ scaling shifted QT/T timing and amplitude in an interpretable direction, but did not by itself provide strong morphology-quality control. The pseudo-diffusion parameter~$\kappa$ acted as a secondary smoothing and modulation parameter: it improved local repolarization smoothness and recovery plausibility with little QTc displacement. This interpretation is more conservative than treating $\kappa$ as a primary T-wave or QT control knob, and is consistent with the normalized sensitivity maps.

The curation results highlight why policy-aware interpretation is necessary for synthetic ECG generation. The throughput screen and final balanced curation experiment answer different questions: the former measures passage through a broad recovery-score filter, whereas the latter enforces multi-lead morphology criteria including Lead~II T-polarity. Under the final curation policy, \textrm{RE} produced a modestly larger accepted set and slightly higher feature-space coverage than \textrm{ET}. The magnitude of this improvement supports a conservative interpretation: \textrm{RE} adds recovery-aware morphology modulation while preserving activation controllability, rather than establishing broad superiority over \textrm{ET}.

The downstream delineation evaluation provides an external check on whether the generated signals contain useful training structure. Synthetic pretraining followed by small-fraction PTB-XL+ fine-tuning reduced seed-to-seed variance in P-, QRS-, and T-wave segmentation compared with real-only training in extreme label scarcity (1\% and 2\% label availability), and no collapse was observed in the tested pretrained runs. Furthermore, pretraining with 12 biophysically motivated parameter configurations representing NSR, LBBB, and RBBB-like morphologies achieved small additional accuracy gains over random parameter pretraining in the most label-scarce settings. This is consistent with the intended role of the simulator: while synthetic ECGs from a simplified heart--torso model cannot capture the full diversity of clinical cohorts or replace real annotations, they provide a stabilizing structural initialization that mitigates training failure when labeled clinical data is severely limited.

Several limitations should be noted. The accepted ECGs are not diagnostic-grade clinical ECGs; validation here focuses on internal morphology criteria, activation consistency, and feature-space coverage rather than cohort-level clinical realism. Detector-derived intervals and amplitudes are simulation diagnostics, and the threshold-based acceptance policies should be evaluated against learned or expert-adjudicated curation rules. The graph, torso, and source models are also simplified relative to subject-specific electrophysiology and volume-conductor models. Finally, \textrm{RE} improves curated yield and coverage only modestly and is slower than \textrm{ET}, so its value depends on whether recovery-aware morphology control is needed for the target application.

Future work should evaluate the curated signals against external ECG cohorts, downstream tasks, learned or expert-adjudicated curation policies, and more subject-specific graph and torso models. Another natural direction is inverse ECG modeling: using observed body-surface ECGs to infer latent activation or recovery structure on the cardiac graph. The Bellman certificate and diagnostics-aware morphology checks could serve as constraints or initialization signals in that setting, but the present study restricts its claims to forward synthesis and curation.

% ======================================================================
\section{Conclusion}
\label{sec:conclusion}
% ======================================================================

We presented a diagnostics-aware framework for synthetic ECG generation on a unified cardiac graph. The framework combines a Bellman residual certificate for graph-based activation consistency, \textrm{ET} and \textrm{RE} synthesis backends with separated activation and recovery controls, and a policy-based curation pipeline for ECG morphology assessment. Experiments showed that \textrm{RE}-derived activation remained close to the eikonal backbone, that \textrm{RE} preserved \textrm{ET}-like activation knob responses, and that recovery controls produced interpretable QT/T and morphology effects.

Under the final balanced multi-lead curation policy, the \textrm{RE} backend modestly improved accepted morphology yield and feature-space coverage relative to \textrm{ET} while maintaining activation controllability. These results support the use of certificate-checked activation and diagnostics-aware curation as practical components of curated synthetic ECG generation. Further validation against clinical ECG distributions and downstream signal-processing tasks is needed before clinical representativeness can be claimed.

\bibliography{references}

\clearpage
\section*{Supplementary Material}
\renewcommand{\thesection}{S\arabic{section}}
\renewcommand{\thesubsection}{S\arabic{section}.\arabic{subsection}}
\renewcommand{\thetheorem}{S\arabic{theorem}}
\renewcommand{\theequation}{S\arabic{equation}}
\renewcommand{\thefigure}{S\arabic{figure}}
\renewcommand{\thetable}{S\arabic{table}}
\renewcommand{\theHsection}{supp.\arabic{section}}
\renewcommand{\theHsubsection}{supp.\arabic{section}.\arabic{subsection}}
\renewcommand{\theHtheorem}{supp.\arabic{theorem}}
\renewcommand{\theHequation}{supp.\arabic{equation}}
\renewcommand{\theHfigure}{supp.\arabic{figure}}
\renewcommand{\theHtable}{supp.\arabic{table}}
\setcounter{section}{0}
\setcounter{subsection}{0}
\setcounter{theorem}{0}
\setcounter{equation}{0}
\setcounter{figure}{0}
\setcounter{table}{0}

\section{Proofs of main results}
\label{sec:supp_proofs}

For convenience, let
\[
X
:=
\{x:V\to\mathbb{R}\cup\{\infty\}: x(s)=0 \text{ for all } s\in S\},
\]
let $N:=|V|$,
let $F$ denote the Bellman operator in \eqref{eq:bellman_op}, and let $T$ denote the shortest-path activation time in \eqref{eq:eikonal_time}. For $x\in X$ and $i\notin S$, define the greedy predecessor set
\[
\Pi_x(i):=\arg\min_{j\sim i}\bigl(x(j)+w_{ji}\bigr),
\]
and choose $\pi_x(i)\in\Pi_x(i)$ by a fixed tie-breaking rule.

\subsection{Basic properties of the Bellman operator}

We first record three basic properties of the Bellman operator that are used throughout this supplement.

\begin{lemma}[Monotonicity]
\label{lem:supp_mono}
For any $x,y\in X$, if $x\le y$ pointwise on $V$, then $Fx\le Fy$ pointwise on $V$.
\end{lemma}

\begin{proof}
For $i\in S$, we have $(Fx)(i)=(Fy)(i)=0$.
Now fix $i\notin S$. Since $x(j)\le y(j)$ for all neighbors $j\sim i$, we also have
\[
x(j)+w_{ji}\le y(j)+w_{ji}
\qquad\text{for all } j\sim i.
\]
Taking the minimum over $j\sim i$ gives
\[
(Fx)(i)\le (Fy)(i).
\]
This proves the claim.
\end{proof}

\begin{lemma}[Nonexpansiveness in $\ell^\infty$]
\label{lem:supp_nonexp}
For any $x,y\in X$,
\[
\|Fx-Fy\|_\infty \le \|x-y\|_\infty.
\]
\end{lemma}

\begin{proof}
For $i\in S$, the difference is zero.
Now fix $i\notin S$, and define
\[
a_j:=x(j)+w_{ji},
\qquad
b_j:=y(j)+w_{ji},
\qquad j\sim i.
\]
Then
\[
|(Fx)(i)-(Fy)(i)|
=
\left|\min_{j\sim i} a_j - \min_{j\sim i} b_j\right|
\le
\max_{j\sim i}|a_j-b_j|.
\]
Since $a_j-b_j=x(j)-y(j)$, it follows that
\[
|(Fx)(i)-(Fy)(i)|
\le
\max_{j\sim i}|x(j)-y(j)|
\le
\|x-y\|_\infty.
\]
Taking the maximum over $i\in V$ yields the result.
\end{proof}

\begin{lemma}[Path estimate under an approximate DPP]
\label{lem:supp_path_estimate}
Let $x\in X$ and $\delta\ge 0$.
If
\[
x\le Fx+\delta
\qquad\text{pointwise on }V,
\]
then for any path $\gamma=(v_0,\dots,v_m)$ with $v_0\in S$ and $v_m=i$,
\[
x(i)\le \operatorname{cost}_w(\gamma)+m\delta.
\]
\end{lemma}

\begin{proof}
For each $k=1,\dots,m$, the assumption $x\le Fx+\delta$ gives
\[
x(v_k)\le (Fx)(v_k)+\delta
=
\min_{u\sim v_k}\bigl(x(u)+w_{uv_k}\bigr)+\delta
\le x(v_{k-1})+w_{v_{k-1}v_k}+\delta.
\]
Summing these inequalities over $k=1,\dots,m$ and using $x(v_0)=0$ yields
\[
x(i)=x(v_m)\le \sum_{k=1}^m w_{v_{k-1}v_k}+m\delta
=\operatorname{cost}_w(\gamma)+m\delta.
\]
\end{proof}

\subsection{Proof of Theorem~\ref{thm:bellman_fp}}

\begin{proof}[Proof of Theorem~\ref{thm:bellman_fp}]
We prove first that $T=FT$.

If $i\in S$, then by definition $T(i)=0=(FT)(i)$.

Now fix $i\notin S$.

\medskip
\noindent
\textit{Step 1: $T(i)\le (FT)(i)$.}
Let $j\sim i$ be any neighbor.
Take a path $\gamma_j$ from $S$ to $j$ that achieves $T(j)$.
Then $\gamma_j$ followed by the edge $(j,i)$ is a path from $S$ to $i$ of cost $T(j)+w_{ji}$.
Hence
\[
T(i)\le T(j)+w_{ji}
\qquad\text{for all } j\sim i.
\]
Taking the minimum over neighbors gives
\[
T(i)\le \min_{j\sim i}\bigl(T(j)+w_{ji}\bigr)=(FT)(i).
\]

\medskip
\noindent
\textit{Step 2: $T(i)\ge (FT)(i)$.}
Let $\gamma=(v_0,\dots,v_m)$ be a shortest path from some $v_0\in S$ to $v_m=i$ that achieves $T(i)$.
Then $v_{m-1}\sim i$, and
\[
T(i)=\operatorname{cost}_w(\gamma)
=
\operatorname{cost}_w(v_0,\dots,v_{m-1})+w_{v_{m-1}i}
\ge
T(v_{m-1})+w_{v_{m-1}i}.
\]
Therefore
\[
T(i)\ge \min_{j\sim i}\bigl(T(j)+w_{ji}\bigr)=(FT)(i).
\]

Combining the two steps yields $T=FT$.

\medskip
\noindent
We now prove uniqueness.
Let $x\in X$ satisfy $x=Fx$.

First, we show $x\le T$.
Fix $i\in V$, and let $\gamma=(v_0,\dots,v_m)$ be any path from $v_0\in S$ to $v_m=i$.
Since $x=Fx$,
\[
x(v_k)=(Fx)(v_k)\le x(v_{k-1})+w_{v_{k-1}v_k}
\qquad\text{for } k=1,\dots,m.
\]
Summing over $k$ and using $x(v_0)=0$ gives
\[
x(i)\le \operatorname{cost}_w(\gamma).
\]
Taking the minimum over all such paths yields $x(i)\le T(i)$.

Next, we show $x\ge T$.
Fix $i\notin S$, and choose a greedy predecessor $\pi_x(i)\in\Pi_x(i)$.
Since $x=Fx$,
\[
x(i)=x(\pi_x(i))+w_{\pi_x(i)i}.
\]
Because $w_{\pi_x(i)i}>0$, we have
\[
x(\pi_x(i))=x(i)-w_{\pi_x(i)i}<x(i).
\]
Thus $x$ strictly decreases along the predecessor chain, so the chain cannot contain a directed cycle.
Since $V$ is finite, the chain reaches the source set in at most $N-1$ steps.
Therefore there exist nodes $v_0\in S, v_1,\dots,v_m=i$ such that
\[
v_{k-1}=\pi_x(v_k), \qquad k=1,\dots,m.
\]
Summing the equalities along this chain gives
\[
x(i)=x(v_0)+\sum_{k=1}^m w_{v_{k-1}v_k}
=
\operatorname{cost}_w(v_0,\dots,v_m).
\]
Since $T(i)$ is the minimum path cost from $S$ to $i$, we obtain $T(i)\le x(i)$.

Hence $x=T$ on $V$, which proves uniqueness.
\end{proof}

\subsection{Proof of Theorem~\ref{thm:comparison}}

When the greedy predecessor graph is acyclic, every reachable node $i$ has a well-defined greedy depth
\[
m(i):=
\begin{cases}
0, & i\in S,\\
1+m(\pi_x(i)), & i\notin S,
\end{cases}
\qquad
m_{\max}:=\max_{i\in V} m(i).
\]

\begin{proof}[Proof of Theorem~\ref{thm:comparison}]
Assume $\delta_{\mathrm{Bell}}(x)\le\delta$.
Then for every $i\in V$,
\[
-\delta\le x(i)-(Fx)(i)\le \delta.
\]
Equivalently, for every $i\notin S$,
\[
\bigl|x(i)-\bigl(x(\pi_x(i))+w_{\pi_x(i)i}\bigr)\bigr|\le \delta.
\]

Fix a reachable node $i\in V$, and let
\[
v_0\in S,\; v_1,\dots,v_{m(i)}=i
\]
be the greedy predecessor chain, so that $v_{k-1}=\pi_x(v_k)$ for $k=1,\dots,m(i)$.
Summing the one-step inequalities along this chain and using $x(v_0)=0$ gives
\[
\left|x(i)-\sum_{k=1}^{m(i)} w_{v_{k-1}v_k}\right|
\le
m(i)\delta.
\]

Using the same telescoping argument as in the proof of Theorem~\ref{thm:bellman_fp}, the greedy chain yields the nodewise estimate
\[
|x(i)-T(i)|\le m(i)\delta.
\]
Therefore
\[
\|x-T\|_\infty \le m_{\max}\delta.
\]
\end{proof}

\begin{remark}[Relation to the worst-case factor]
Since $m(i)\le N-1$ for every reachable node on an acyclic greedy predecessor graph, the coarser worst-case form
\[
\|x-T\|_\infty \le (N-1)\delta
\]
is recovered immediately. In the experiments, however, $m_{\max}$ was much smaller than $N-1$, so the greedy-depth form is substantially more informative numerically.
\end{remark}

\subsection{Affine calibration and scaled Bellman comparison}

% The affine calibration in Section~\ref{sec:theory} is obtained by regressing a simulator-produced activation field $\tau$ against the Eikonal time $T$ on the set of reachable nodes
% \[
% \mathcal{R}:=\{i\in V: T(i)<\infty\}.
% \]

% \begin{proposition}[Least-squares affine calibration]
% \label{prop:supp_affine_fit}
% Assume $|\mathcal{R}| \ge 2$ and that $T$ is not constant on $\mathcal{R}$. The minimizer of
% \[
% Q(\alpha,\beta)
% :=
% \sum_{i\in\mathcal{R}}
% \bigl(\tau(i)-(\alpha T(i)+\beta)\bigr)^2
% \]
% is given by
% \[
% \widehat\alpha
% =
% \frac{\sum_{i\in\mathcal{R}} (T(i)-\overline T)(\tau(i)-\overline\tau)}
% {\sum_{i\in\mathcal{R}} (T(i)-\overline T)^2},
% \qquad
% \widehat\beta
% =
% \overline\tau-\widehat\alpha\,\overline T,
% \]
% where
% \[
% \overline T := \frac{1}{|\mathcal{R}|}\sum_{i\in\mathcal{R}} T(i),
% \qquad
% \overline\tau := \frac{1}{|\mathcal{R}|}\sum_{i\in\mathcal{R}} \tau(i).
% \]
% \end{proposition}

% \begin{proof}
% Setting the partial derivatives of $Q$ to zero gives the normal equations
% \[
% \sum_{i\in\mathcal{R}} \bigl(\tau(i)-\alpha T(i)-\beta\bigr)=0,
% \qquad
% \sum_{i\in\mathcal{R}} T(i)\bigl(\tau(i)-\alpha T(i)-\beta\bigr)=0.
% \]
% The first equation yields $\beta=\overline\tau-\alpha\overline T$. Substituting this into the second equation gives
% \[
% \sum_{i\in\mathcal{R}} (T(i)-\overline T)(\tau(i)-\overline\tau)
% \;=\;
% \alpha
% \sum_{i\in\mathcal{R}} (T(i)-\overline T)^2,
% \]
% from which the stated formula for $\widehat\alpha$ follows. The formula for $\widehat\beta$ is then immediate.
% \end{proof}

\begin{lemma}[Homogeneity of the scaled Bellman operator]
\label{lem:supp_scaled_homogeneity}
For $\alpha>0$ and any $y\in X$,
\[
F_{\alpha w}(\alpha y)=\alpha F(y),
\]
where $F_{\alpha w}$ is the Bellman operator with scaled weights $\alpha w_{ij}$.
\end{lemma}

\begin{proof}
For $i\in S$, both sides are zero.
For $i\notin S$,
\[
(F_{\alpha w}(\alpha y))(i)
=
\min_{j\sim i}\bigl(\alpha y(j)+\alpha w_{ji}\bigr)
=
\alpha \min_{j\sim i}\bigl(y(j)+w_{ji}\bigr)
=
\alpha(Fy)(i).
\]
\end{proof}

\begin{proposition}[Affine comparison on the scaled graph]
\label{prop:supp_affine_comparison}
Let $\tau:V\to\mathbb{R}$ satisfy $\tau(s)=\beta$ for all $s\in S$, and set $\tilde\tau:=\tau-\beta\in X$. Fix $\alpha>0$ and assume
\[
\delta_{\mathrm{affBell}}(\tau;\alpha,\beta)
=
\|\tilde\tau-F_{\alpha w}(\tilde\tau)\|_\infty
\le \delta.
\]
If the greedy predecessor graph of $\tilde\tau$ with respect to $F_{\alpha w}$ is acyclic, then
\[
\|\tau-(\alpha T+\beta)\|_\infty \le m_{\max}\delta,
\]
where $m_{\max}$ is the maximum greedy depth of $\tilde\tau$ with respect to the scaled operator $F_{\alpha w}$.
\end{proposition}

\begin{proof}
By Lemma~\ref{lem:supp_scaled_homogeneity} and Theorem~\ref{thm:bellman_fp},
\[
F_{\alpha w}(\alpha T)=\alpha F(T)=\alpha T,
\]
so $\alpha T$ is the fixed point of the scaled Bellman operator. Applying Theorem~\ref{thm:comparison} to the normalized field $\tilde\tau$ on the graph with weights $\alpha w$ gives
\[
\|\tilde\tau-\alpha T\|_\infty \le m_{\max}\delta.
\]
Because $\tilde\tau=\tau-\beta$,
\[
\|\tau-(\alpha T+\beta)\|_\infty
=
\|\tilde\tau-\alpha T\|_\infty,
\]
which proves the claim.
\end{proof}

In practice, we first compute $(\widehat\alpha,\widehat\beta)$ by least-squares regression of the simulator-produced activation field against the Eikonal time, and then evaluate the scaled residual
\[
\delta_{\mathrm{affBell}}(\tau;\widehat\alpha,\widehat\beta)
=
\bigl\|(\tau-\widehat\beta)-F_{\widehat\alpha w}(\tau-\widehat\beta)\bigr\|_\infty.
\]
This separates a global timing-scale mismatch from the remaining local deviations, which is why the affine residual is the primary activation certificate reported in the heart-graph experiments.

\section{Diagnostic Metric Definitions}
\label{sec:supp_metrics}

This section gives the operational definitions used by the Stage~1 diagnostics pipeline. The main manuscript reports these metrics at the level needed to interpret the curation policies, while the implementation records the underlying activation, recovery, waveform, and policy metadata for each generated sample.

\subsection{Mechanism-aware interval metrics}

For each node~$i$, let $T_i$ denote the activation time from the graph activation backbone or from the extracted activation marker of the simulated voltage trace. Let $R_i$ denote the local recovery time, defined as the time at which the repolarizing voltage trace crosses the implementation-specific recovery threshold after activation. Let $y_\ell(t)$ denote the simulated ECG in lead~$\ell$.

The interval metrics combine latent graph timings with ECG waveform landmarks. Ventricular activation timing is summarized from the ventricular node set~$V_{\mathrm{vent}}$, and atrial activation timing from the atrial node set~$V_{\mathrm{atr}}$:
\begin{align}
  T_{\mathrm{vent,on}} &= \operatorname{quantile}_{0.05}\{T_i: i\in V_{\mathrm{vent}}\},\\
  T_{\mathrm{vent,off}} &= \operatorname{quantile}_{0.95}\{T_i: i\in V_{\mathrm{vent}}\},\\
  T_{\mathrm{atr,on}} &= \operatorname{quantile}_{0.05}\{T_i: i\in V_{\mathrm{atr}}\}.
\end{align}
The QRS duration proxy is the ventricular activation spread
\begin{equation}
  d_{\mathrm{QRS}} = T_{\mathrm{vent,off}} - T_{\mathrm{vent,on}},
\end{equation}
with ECG waveform landmarks used to reject cases in which the simulated QRS complex is missing, flat, or globally inverted. The PR-related timing proxy is computed from atrial-to-ventricular activation delay,
\begin{equation}
  d_{\mathrm{PR}} = T_{\mathrm{vent,on}} - T_{\mathrm{atr,on}},
\end{equation}
and is interpreted within each experiment rather than as an absolute clinical PR interval.

QT and T-wave metrics use both recovery summaries and lead-wise ECG landmarks. A recovery-based QT proxy is computed as
\begin{equation}
  d_{\mathrm{QT,rec}} =
  \operatorname{quantile}_{0.95}\{R_i: i\in V_{\mathrm{vent}}\}
  - T_{\mathrm{vent,on}}.
\end{equation}
Lead-wise T-wave landmarks are then extracted from the post-QRS window of $y_\ell(t)$, including T-peak time, T-end time, and T-wave amplitude. The reported QT, QTc, and $T_{\mathrm{peak}}$--$T_{\mathrm{end}}$ values use these ECG landmarks when they are stable and use the recovery-time summaries as consistency checks. QTc is obtained from the same QT estimate using the fixed simulated cycle length for the corresponding experiment.

This hybrid definition is specific to the simulation setting: it uses latent activation and recovery fields that are generated by the mechanistic model and are therefore not directly available to a purely waveform-generative model.

\subsection{Waveform morphology metrics}

Waveform morphology metrics are computed from the simulated 12-lead ECG after baseline alignment and common timing-window selection. For each analyzed lead, the detector records P-wave amplitude, QRS amplitude, R-peak polarity, T-wave amplitude, T-wave polarity, and the presence or absence of the expected P-QRS-T ordering. Samples are flagged for gross morphology failure when they contain flatline-like traces, missing QRS complexes, global QRS inversion, implausible wave ordering, or unstable T-wave landmarks.

Lead-specific checks are policy dependent. For example, Lead~II T-wave polarity is used by the final balanced morphology-curation policy, whereas broader recovery-screening policies may accept samples that are useful for recovery-mechanism analysis but not for final curated ECG morphology.

\subsection{Recovery scores}

The recovery plausibility score $s_{\mathrm{rep}}$ is the policy-level screening score reported in the recovery-atlas summary table. It is stored in the experiment outputs as \texttt{recovery\_plausibility\_score}. Each sample receives normalized component scores for transmural recovery-gradient consistency, QT/T-window plausibility, T-wave presence and ordering, and local smoothness of neighboring recovery times. The aggregate score is computed as
\begin{equation}
  s_{\mathrm{rep}} =
  \frac{\sum_m w_m q_m}{\sum_m w_m},
\end{equation}
where $q_m\in[0,100]$ denotes a normalized component score and $w_m$ denotes the corresponding policy weight. Larger values indicate more plausible recovery behavior under the selected policy. Samples below the policy threshold are rejected with the \texttt{balanced\_recovery\_score} flag in the curation logs.

The repolarization score is distinct from $s_{\mathrm{rep}}$. It is reported as a local diagnostic rather than as the primary acceptance score. It emphasizes smoothness of neighboring recovery times on the myocardial graph and is used to interpret the effect of the diffusion strength~$\kappa$ on local recovery regularity. Thus, recovery plausibility is the policy-level score used for screening, whereas the repolarization score is a mechanism-oriented smoothness diagnostic.

\subsection{Feature-space coverage}

Feature-space coverage quantifies diversity among accepted ECGs. For the final curation experiment, selected interval, amplitude, and polarity features are discretized into fixed bins over the admissible feature range. Coverage is the fraction of occupied bins,
\begin{equation}
  C = \frac{\#\{\text{occupied admissible bins}\}}
           {\#\{\text{admissible bins}\}}.
\end{equation}
Per-model coverage is computed separately for \textrm{ET} and \textrm{RE} under the same sampling budget. Pooled coverage is computed after combining accepted samples from both models and is used only as a summary of the joint accepted feature space.

\end{document}